%% file: centroids_arxiv.tex
\documentclass[11pt]{article}
\newcommand{\comment}[1]{{$\star$\sf\textbf{#1}$\star$}}

\usepackage{graphicx,color}
\usepackage{hangcaption}
\usepackage{amssymb}
\usepackage{amsmath}
\usepackage{amsthm}
\usepackage{subfigure}
\usepackage{bm}
\usepackage{enumerate}

\theoremstyle{plain}
\newtheorem{thm}{Theorem}[section]
\newtheorem{cor}[thm]{Corollary}
\newtheorem{lem}[thm]{Lemma}
\newtheorem{prop}[thm]{Proposition}

\theoremstyle{definition}
\newtheorem{defn}[thm]{Definition}

\newtheorem{ex}[thm]{Example}
\newtheorem{rem}[thm]{Remark}
\newtheorem{alg}[thm]{Algorithm}

\newtheorem*{note*}{Important note}

\newenvironment{bullets}{\begin{itemize}\setlength\itemsep{-.25ex}}{\end{itemize}}

\newenvironment{alglist}%
    {\begin{list}
        {}
        {\leftmargin=5.4em\labelwidth=5.1em\labelsep=.6em
         \topsep=-2ex\itemsep=-.5ex}}
    {\vspace{1ex}\end{list}}
\def\routine#1{\item[{\sc{#1}{\ }}]}%\dotfill\textsc{#1}}
\def\procedure#1{{\sc{#1}}}
\newenvironment{routinelist}[1]%
    {\routine{#1}\begin{list}
        {}
        {\leftmargin=2.9em\labelwidth=2.4em\labelsep=.5em
         \topsep=-2ex\itemsep=-.5ex}}
    {\end{list}}
    {\begin{list}
        {}
        {\leftmargin=3.0em\labelwidth=4.8em\labelsep=.5em
         \topsep=-2ex\itemsep=-.5ex}}
    {\end{list}}

\newcounter{separated}

\newcommand\RR{\mathbb{R}}

\newcommand\cA{\mathcal{A}}
\newcommand\cB{\mathcal{B}}

\newcommand\cE{\mathcal{E}}

\newcommand\cG{\mathcal{G}}
\newcommand\cM{\mathcal{M}}
\newcommand\Or{\mathcal{O}}
\newcommand\cP{\mathcal{P}}
\newcommand\cS{\mathcal{S}}
\newcommand\cT{\mathcal{T}}

\newcommand\cV{\mathcal{V}}
\newcommand\cX{\mathcal{X}}

\newcommand\oH{\hspace{.35ex}\ol{\hspace{-.35ex}H}}
\newcommand\ol[1]{{\overline{#1}}}
\newcommand\tx{\textstyle}
\newcommand\wt[1]{{\widetilde{#1}}}
\newcommand\del{\partial}
\newcommand\geod{\gamma}

\newcommand\minus{\setminus}

\newcommand\nothing{\varnothing}
\newcommand\tr[1]{ \begin{list}{}{\setlength{\leftmargin}{#1em}} \item}

\newcommand\tl{ \end{list}}

\DeclareMathOperator{\var}{var}

\providecommand{\abs}[1]{\lvert#1\rvert}
\providecommand{\norm}[1]{\lVert#1\rVert}

\graphicspath{{./}{figs/}}

\begin{document}%%%%%%%%%%%%%%%%%%%%%%%%%%%%%%%%%%%%%%%%%%%%%%%%%%%%%%
%%%%%%%%%%%%%%%%%%%%%%%%%%%%%%%%%%%%%%%%%%%%%%%%%%%%%%%%%%%%%%%%%%%%%%
%%%%%%%%%%%%%%%%%%%%%%%%%%%%%%%%%%%%%%%%%%%%%%%%%%%%%%%%%%%%%%%%%%%%%%

%\begin{frontmatter}

\title{\vspace{-5ex}Polyhedral computational geometry for averaging metric phylogenetic trees\vspace{-.5ex}}
%\title{Averaging metric phylogenetic trees}
%\title{\vspace{-5ex}Averaging metric phylogenetic trees\vspace{-.5ex}}
\author{Ezra Miller \and Megan Owen \and J. Scott Provan}
%\address{Department of Mathematics, Duke University, Durham, NC, 27708, USA}
%\ead{ezra@math.duke.edu}
%\author{Megan Owen\corref{cor1}}
%\ead{megan.owen@uwaterloo.ca}
%\address{Cheriton School of Computer Science, University of Waterloo, Waterloo, ON, N2L 3G1, Canada, 1-519-888-4567 ex 33884}
%\author{J. Scott Provan}
%\ead{scott_provan@unc.edu}
%\address{Department of Statistics and Operations Research, University of North Carolina, Chapel Hill, NC, 27599-3180, USA}

%\cortext[cor1]{Corresponding author}

%\date{Received: date / Accepted: date}
%\date{22 December 2012\vspace{-.5ex}}
\maketitle

\begin{abstract}
This paper investigates the computational geometry relevant to
calculations of the Fr\'echet mean and variance for probability
distributions on the phylogenetic tree space of Billera, Holmes and
Vogtmann, using the theory of probability measures on spaces of
nonpositive curvature developed by Sturm.  We show that the
combinatorics of geodesics with a specified fixed endpoint in tree
space are determined by the location of the varying endpoint in a
certain polyhedral subdivision of tree space.  The variance function
associated to a finite subset of tree space has a fixed $C^\infty$
algebraic formula within each cell of the corresponding subdivision,
and is continuously differentiable in the interior of each orthant of
tree space.  We use this subdivision to establish two iterative
methods for producing sequences that converge to the Fr\'echet mean:
one based on Sturm's Law of Large Numbers, and another based on
descent algorithms for finding optima of smooth functions on convex
polyhedra.  We present properties and biological applications of
Fr\'echet means and extend our main results to more general globally
nonpositively curved spaces composed of Euclidean orthants.
\end{abstract}

%\begin{keyword}
%tree space \sep Fr\'echet mean \sep polyhedral subdivision \sep descent method \sep phylogenetics \sep nonpositively curved space
%\end{keyword}
%\end{frontmatter}

%
%\subclass{68U05 \and 05C0 \and 52B99 \and 62P10 \and 92D15 \and 60B05 \and 52-04 \and 62H99 \and 52A41 \and 90C57 \and 92B10 \and  92-08 \and 53C23}

\setcounter{tocdepth}{3}
\tableofcontents

%\newpage
%\pagenumbering{arabic}

%%%%%%%%%%%%%%%%%%%%%%%%%%%%%%%%%%%%%%%%%%%%%%%%%%%%%%%%%%%%%%%%%%%%%%
\section*{Introduction}%%%%%%%%%%%%%%%%%%%%%%%%%%%%%%%%%%%%%%%%%%%%%%%
\addcontentsline{toc}{section}{\numberline{}Introduction}%%%%%%%%%%%%%
%%%%%%%%%%%%%%%%%%%%%%%%%%%%%%%%%%%%%%%%%%%%%%%%%%%%%%%%%%%%%%%%%%%%%%

The development of statistical methods for studying phylogenetic
trees, and in particular the search for meaningful notions of
consensus tree for phylogenetic data, has been of considerable
importance in biology for four decades.  Starting with the problem as
posed by Adams~\cite{Adams}, a great deal of research has been done,
and a myriad of definitions proposed, relating to consensus trees in
phylogenetics; see \cite{Bryant03} for an excellent overview.
The problem has been confounded by the combinatorial nature of the
trees themselves.
According to Cranston and Rannala \cite{bayesPosterior},
``Phylogenetic inference has long been troubled by the difficulty of
performing statistical analysis on tree topologies. The topologies are
discrete, categorical, and non-nested hypotheses about the species
relationships.  They are not amenable to standard summary analyses
such as the calculation of means and variances and cause difficulties
for many traditional forms of hypothesis testing.''  Other papers
share concerns about issues such as these \cite{Barret,Holmes}.
%HolderEtAl08}.

The introduction by Billera, Holmes, and Vogtmann of phylogenetic tree
space \cite{BHV01} opened statistical analysis of tree-like data to a
wide and computationally tractable variety of techniques
\cite{Holmes05}.  Tree space, with its geodesic distance, is a
\emph{globally nonpositively curved}\/ (abbreviated to \emph{global
NPC}) space, and as a result it has convexity properties that imply
uniqueness of means as well as other important statistical and
geometric objects, while also giving a framework for effective
computational methods to calculate these objects.  One of the major
uses of the convexity properties was the discovery by Owen and
Provan~\cite{OwenProvan10} of a fast algorithm for computing geodesics
in this space (see Section~\ref{s:treespace} for this algorithm as
well as the background tree space geometry necessary to state it).
Chakerian and Holmes \cite{chakerian-holmes} subsequently showed that
phylogenetic tree space provides an excellent platform for
implementing several distance-based statistical techniques, and Nye
\cite{Nye12} has shown how this space can be used to perform principal
component analysis on tree data.

Perhaps the two most fundamental concepts of interest in statistical
analysis of data are that of \emph{sample mean} (or \emph{average})
and its associated \emph{variance}.  The basic goal of this paper is
to
demonstrate the computational effectiveness of certain notions of
statistical mean and variance for probability distributions on tree
space.  The average that we use is the \emph{Fr\'echet mean}, or
\emph{barycenter}: the point in tree space that minimizes its sum of
squared geodesic distances to the sample points
(Section~\ref{s:mean}).  Our decision to use this definition is
motivated by work of Sturm~\cite{sturm}, who identified the Fr\'echet
mean as a theoretically rich statistical object associated with
sampling from a specified distribution on a global NPC space (see
Theorem~\ref{t:sturmPoints}).
Fr\'echet means in tree space and the algorithm for computing them
that arises from Sturm's work (Algorithm~\ref{a:sturm}) have been
independently developed~by~Ba\v c\'ak~\cite{Bacak12}.

Our principal theoretical contribution lies in the
discovery of polyhedral structure governing the variation of geodesics
in tree space as one endpoint varies (Section~\ref{s:vistal}).  To be
more precise,
if $T$ is a fixed point in tree space, then in appropriate coordinates
on tree space, the set of
points whose geodesics to~$T$ share the same combinatorics comprise a
convex polyhedral cone called a \emph{vistal cell}
(Theorem~\ref{t:vistal_face_combinatorics}), and the vistal cells
constitute a polyhedral subdivision of tree space
(Theorem~\ref{t:cell_complex}).  
This \emph{metric combinatorics} also
arises in single source shortest path queries (see \cite{Mitchell00}
for a survey), and has direct roots in surprisingly similar statements for boundaries of
convex polyhedra \cite{unfolding}, the parallel being unexpected
because boundaries of convex polyhedra are positively curved, in
contrast to the negative curvature of tree space.
However, polyhedrality of the subdivision is generally not encountered
  outside of the planar or positively curved cases, and thus is
  completely unexpected here; see Example~\ref{e:4dim} for a hint of
  the complexity that can occur even for global NPC cubical complexes.

Metric combinatorics of tree space, particularly its polyhedral
nature, combines with generalities on nonlinear optimization in NPC
spaces to give a second iterative method converging to the mean
(Algorithm~\ref{a:descent}) via descent
procedures.  The crucial observations are that the variance function
has a unique local minimum on tree space, is continuously
differentiable on each Euclidean orthant in tree space, and has a simple
algebraic formula within the interior of each~vistal~cell.

% describe the problem: how to find means in tree space; mention that
% the long term goal is to do statistics on tree space, where we do
% flats, etc (but for future paper) \comment{EM: I don't think it's
% wise to mention anything as specific as flats unless we know we have
% specific work in progress on that}

Means in tree space have subtle, sometimes peculiar properties that
inform our particular motivations (Section~\ref{s:apps}), which come
primarily from biological and medical applications, although we expect
these observations to impact other fields where distributions of
metric trees naturally appear.  Evolutionary biology, for instance,
considers actual phylogenetic trees, each representing a putative
evolutionary history of a set species or genes
(Example~\ref{e:gene}).  In medical imaging, trees can represent blood
vessels in human brain scans \cite{SkwererJMIV13}  or lung airway trees \cite{IPMI2013}, for example.

Some of the theory in Sections~\ref{s:treespace}\textendash{}\ref{s:computing}
extends to arbitrary global NPC spaces, and all of it extends to
global NPC orthant spaces (Section~\ref{s:npc}).  For the first
iterative procedure (Algorithm~\ref{a:sturm}) and the rest of
Section~\ref{s:mean}, as well as for the shortest path combinatorics
in Section~\ref{s:treespace}, this means working in arbitrary global
NPC spaces (Sections~\ref{ss:geomNPC}\textendash{}\ref{ss:NPCmean}).  For the
second iterative procedure (Algorithm~\ref{a:descent}) and the rest of
Section~\ref{s:computing}, as well as for the metric combinatorics in
Section~\ref{s:vistal}, this means working in piecewise Euclidean
global NPC spaces that are formed by gluing orthants together by rules
similar to \textemdash{} but substantially more general than \textemdash{} those defining tree
space (Section~\ref{sub:orthant}).  The extensions suggest exciting
new research in applying both statistical methods and numerical
nonlinear programming techniques to a wide variety of problems.
\textbf{Important note:} readers interested in the generality of
abstract orthant spaces or arbitrary NPC spaces are urged to begin
with Section~\ref{s:npc}, which sets up the notation and concepts in
Sections~\ref{s:treespace}--\ref{s:computing} from that perspective.
Hence such readers can avoid checking the proofs in the earlier
sections twice.

%%%%%%%%%%%%%%%%%%%%%%%%%%%%%%%%%%%%%%%%%%%%%%%%%%%%%%%%%%%%%%%%%%%%%%
\paragraph*{Acknowledgements}%%%%%%%%%%%%%%%%%%%%%%%%%%%%%%%%%%%%%%%%
%%%%%%%%%%%%%%%%%%%%%%%%%%%%%%%%%%%%%%%%%%%%%%%%%%%%%%%%%%%%%%%%%%%%%%

Our thanks go to Michael Turelli and Elen Oneal for help with
references and discussions on biological applications, to Antonis
Rokas for kindly providing the yeast data set, 
%to Sean Skwerer and \.Ipek O\u guz for preparing the brain artery dataset, 
and to Dennis
Barden for comments on a draft of the paper.  EM had support from NSF
grants DMS-0449102 = DMS-1014112 and DMS-1001437.  MO was partially
supported by a desJardins Postdoctoral Fellowship in Mathematical
Biology at University of California Berkeley and by the U.S. National
Science Foundation under grant DMS-0635449 to the Statistical and
Applied Mathematical Sciences Institute (SAMSI).  Much of this research was
facilitated by and carried out at SAMSI as an outgrowth of the
2008\textendash2009 program on Algebraic Methods in Systems Biology
and~Statistics.

%%%%%%%%%%%%%%%%%%%%%%%%%%%%%%%%%%%%%%%%%%%%%%%%%%%%%%%%%%%%%%%%%%%%%%
\section{Tree space and the geodesic algorithm}\label{s:treespace}%%%%
%%%%%%%%%%%%%%%%%%%%%%%%%%%%%%%%%%%%%%%%%%%%%%%%%%%%%%%%%%%%%%%%%%%%%%

In this section, we describe the space of phylogenetic trees
introduced by Billera, Holmes, and Vogtmann~\cite{BHV01}, as well as a
distance and characterization of geodesics in this space.

%%%%%%%%%%%%%%%%%%%%%%%%%%%%%%%%%%%%%%%%%%%%%%%%%%%%%%%%%%%%%%%%%%%%%%
\subsection{Phylogenetic tree space}\label{ss:treespace}%%%%%%%%%%%%%%

A \emph{phylogenetic $n$-tree}~$T$, or simply an \emph{$n$-tree}, is
an acyclic graph $T$ with edge set $\cE=\cE_T$ whose leaves (degree 1
nodes) are labeled with index set $L=\{0,1,\ldots,n\}$, and whose
interior vertices have degree at least 3.  (The label 0 is often
referred to as the \emph{root}\/ of $T$, although that is not relevant
in this paper.)  The maximum number of edges in an $n$-tree is $2n-1$.
Each edge $e$ of $T$ is assigned a nonnegative \emph{length} $|e|_T$,
or $\abs e$ in case the ambient tree is clear.  Removal of any edge
$e$ from $T$ determines a unique partition of the leaves of $T$ into
two subsets $X_e$ and $\ol X_e$; the pair $X_e|\ol X_e$ is called the
\emph{split}\/ associated with $e$.  A key property of splits in trees
is that the splits $X_e|\ol X_e$ and $X_f|\ol X_f$ of any pair of
edges $e$ and $f$ are \emph{compatible}, that is, one of the sets $X_e
\cap X_f$, $X_e \cap \ol X_f$, $\ol X_e \cap X_f$, or $\ol X_e \cap
\ol X_f$ is empty.  A set $S$ of splits is called \emph{compatible} if
every pair of splits in $S$ is compatible.  It turns out
\cite[Theorem~3.1.4]{SS03} that any compatible set of splits on $L$
corresponds to a unique tree, and so from now on we identify a tree
$T$ by simply giving the splits and edge lengths for each edge in~$T$.

A tree $T$ can have an edge $e$ whose associated length $|e|_T$
is~$0$.  This corresponds to the edge $e$ having been
\emph{contracted}\/ in~$T$.  Denoting the set of edges of~$T$ with
nonzero length by~$\cE^+_T$ allows the identification $T \sim T'$
between two trees $T$ and~$T'$ whenever (i)~$\cE^+_T=\cE^+_{T'}$ and
(ii)~their nonzero edge lengths are equal.

\begin{ex}\label{e:tree_examples}
Two 5-trees are depicted in Figure~\ref{f:tree_examples}.  For
simplicity, we only give the splits and edge lengths for the three
internal edges in each tree.  The six internal edges are distinct
since they have different splits, and the splits within each tree are
compatible.  The only compatible pairs between the two trees, however,
are $\{e_1,e_6\}$, $\{e_2,e_5\}$, and $\{e_3,e_4\}$.
\begin{figure}[ht]
\centering
\scalebox{0.9}{ \input{tree_examples.pstex_t} }\\[1.5em]
\caption{An example of two 5-trees.}
\label{f:tree_examples}
\centering
\end{figure}
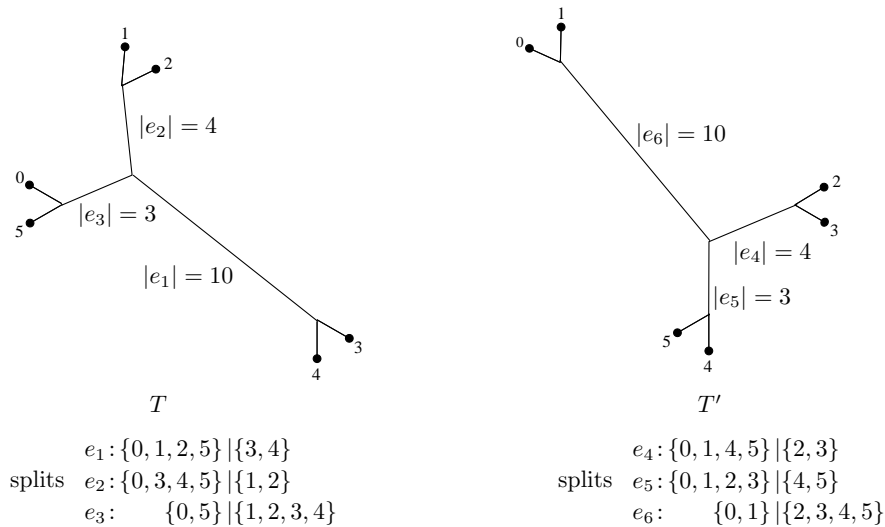
\end{ex}

The \emph{tree space}\/ $\cT_n$ introduced by Billera, Holmes, and
Vogtmann~\cite{BHV01} is the space of all phylogenetic $n$-trees.  It
is obtained by representing each tree $T\in \cT_n$ on edge set $\cE$
by a vector in the Euclidean \emph{orthant}\/ $\Or(T) = \Or(\cE)=
\RR^\cE_+$, whose coordinate values are equal to the corresponding
lengths of the edges of $T$.  As above, trees $T$ and $T'$ are
identified between orthants whenever the associated trees satisfy $T
\sim T'$.  This makes $\cT_n$ a union of $(2n-1)$-dimensional
orthants \textemdash{}called \emph{maximal}\/ orthants \textemdash{}whose interiors are
disjoint and which are identified along their boundaries through the
equivalence $\sim$ given above.  A \emph{path} in $\cT_n$ is the image
of a continuous map $\gamma:[0,1] \to \cT_n$.  The \emph{Euclidean length} of a path
in $\cT_n$ is the sum of the Euclidean lengths of its restrictions to
the maximal orthants.  This length endows $\cT_n$ with the metric~$d$
in which $d(T,T')$ is the infimum of the Euclidean lengths of the
paths from $T$ to~$T'$.  Note that $d(T,T') < \infty$, since the space
$\cT_n$ is path-connected: any two points can be joined by straight
line segments through the origin.

%%%%%%%%%%%%%%%%%%%%%%%%%%%%%%%%%%%%%%%%%%%%%%%%%%%%%%%%%%%%%%%%%%%%%%
\subsection{Geodesics in tree space}\label{ss:geodesics}%%%%%%%%%%%%%%

Billera, Holmes, and Vogtmann \cite{BHV01} show that tree space is
\emph{globally non-positively curved}\/ (a \emph{global NPC} space),
equivalently known in this context as \emph{CAT(0)}.  Among other
things, this implies that shortest paths in tree space are unique, so
they are unambiguously referred to as \emph{geodesics}.
This section summarizes the key results of \cite{Owen10} and
\cite{OwenProvan10}, which investigate the structure of geodesics in
tree space and provide an $O(n^4)$-algorithm \textemdash{} the \emph{GTP
algorithm} \textemdash{}to find shortest paths.  For notation, if $T$ is a tree
with edge set~$\cE$ and $A \subseteq \cE$, then we write
$$%
  \norm A_T = \sqrt{\sum_{e \in A} |e|^2_T}
$$
and use $\norm A$ if the tree $T$ is clear.  This means that $\norm A
= |e|$ whenever $A = \{e\}$.

We express a geodesic with endpoints $X$ and $T$ as a parameterized
curve $\geod: [0,1] \to \cT_n$ with $\geod(0)=X$, $\geod(1)=T$, and 
$d(\geod(t),\geod(t')) = \abs{t - t'} \cdot d(X,T)$ for all $t,t' \in [0,1]$.  If
an edge $e$ lies in both $X$ and~$T$, then it lies in every tree on
the path $\geod$, with length uniformly changing between the two
terminal values \cite[Section~4.2]{BHV01}.  We therefore focus first
on the case when $X$ and $T$ have no internal edges in common, and
ignore the lengths of the \emph{pendant} edges (those containing
leaves) in the distance computation.

Each geodesic in tree space is a sequence of straight line segments,
called \emph{legs}, because tree space is piecewise Euclidean.  Each
leg is contained within a single orthant $\Or(E_i \cup F_i)$, where
$E_i \subseteq \cE_X$ and $F_i \subseteq \cE_T$.  The precise
properties of the sets $E_i$ and $F_i$ making up these legs were
determined in \cite{Owen10}.  In particular, define the \emph{support}
$(\cA,\cB) = ((A_1,\ldots, A_k), (B_1,\ldots, B_k))$ of a geodesic
$\geod$ to consist of a pair consisting of a partition $A_1
\cup\cdots\cup A_k$ of $\cE_X$ and a partition $B_1 \cup \cdots \cup
B_k$ of $\cE_T$ such that the following property holds:
\begin{enumerate}[\quad\rm (P1)]\setlength\itemsep{-.5ex}
\item%(P1)
for each $i > j$, the union $A_i \cup B_j$ is compatible.
\end{enumerate}
The geodesic $\geod$ has legs in $\Or(E_i \cup F_i)$, where
\begin{align*}
  E_i&= A_{i+1}\cup\cdots\cup A_k\\
\text{and }
  F_i&= B_1\cup\cdots\cup B_i.
\end{align*}
The individual pairs $(A_i,B_i)$ are the \emph{support pairs}\/ for
the geodesic.

Whether the shortest piecewise-linear path having these legs actually
forms the geodesic between $X$ and $T$ is determined by the following
two properties for $(\cA,\cB)$.
\begin{enumerate}[\quad\rm (P1)]\addtocounter{enumi}{1}
\item%(P2)
$\displaystyle\frac{\norm{A_1}}{\norm{B_1}}
\leq \frac{\norm{A_2}}{\norm{B_2}}
\leq \cdots
\leq \frac{\norm{A_k}}{\norm{B_k}}$.
This is called the \emph{ratio sequence}\/ for $(\cA,\cB)$.
\item%(P3)
For all $(A_i, B_i)$ and partitions $I_1 \cup I_2$ of $A_i$ and $J_1
\cup J_2$ of $B_i$ such that $I_2\cup J_1$ is compatible, the inequality
$\frac{\norm{I_1}}{\norm{J_1}} \geq \frac{\norm{I_2}}{\norm{J_2}}$
holds.
\end{enumerate}
The properties (P1)\textendash(P3) determine the geodesic between $X$ and $T$,
as well as the algebraic description of this geodesic given in Theorem 2.4 in \cite{OwenProvan10}.

The case where $X$ and $T$ have a nonempty set~$C$ of
common edges was addressed in \cite[Section~4]{OwenProvan10}: remove
the common edges between $X$ and $T$ from each tree, and then find the
paths between the remaining disjoint forests, matching trees by their
leaf sets.  The common edges are then placed into the path with the
length of each such edge being 
\begin{equation}\label{eq:common_edges}
(1-\lambda)|e|_X + \lambda |e|_T.
\end{equation}
This also allows pendant edges to be taken into account.

To be able to work more easily with trees having common edges, we
extend Theorem~2.4 in \cite{OwenProvan10} to the case where $X$ and
$T$ have common edges, and in the process simplify the description of
the geodesic considerably.  To do this, we use the following three
important conventions.
\begin{enumerate}[\quad(a)]\setlength\itemsep{-.5ex}
\item%
An edge is never compatible with itself; thus the pairs of identical
edges in $X$ and~$T$ must appear in the same support pair~$(A_i,
B_i)$.
\item%
$\norm {A_i} = - \sqrt{\sum_{e \in A_i} |e|^2_T}$ for any set $A_i$ of
edges of $X$ in common with $T$.
\item%
We extend the notation for support pair by adding the additional sets
$$%
  A_0 = B_0 = A_{k+1} = B_{k+1} = \nothing
$$
and define $\frac{\norm{A_0}}{\norm{B_0}} = -\infty$
and $\frac{\norm{A_{k+1}}}{\norm{B_{k+1}}} = \infty$.
\end{enumerate}
With these conventions we can restate the unified result.

\begin{thm}\label{t:prelimThm}
Let $X$ and $T$ be any two trees in $\cT_n$ (not necessarily
disjoint), and let $(\cA,\cB)$ be a support for $X$ and $T$ satisfying
(P2) and (P3).  The unique geodesic $\geod =
\{\gamma(\lambda):0\leq\lambda\leq 1\}$ from $X$ to~$T$ has legs
\begin{equation}\label{eq:legs2}
\geod^i = \left\{\gamma(\lambda) : \frac{\norm{A_i}}{\norm{B_i}} \leq
\frac{\lambda}{1-\lambda} < \frac{\norm{A_{i+1}}}{\norm{B_{i+1}}}\right\}
\quad\text{ for } i = 0,\ldots,k,
\end{equation}
The points on each leg $\geod^i$ are associated with the tree $T_i$
having edge set
$$%
  B_1\cup\cdots\cup B_i\cup A_{i+1}\cup\cdots\cup A_k
$$
and edge lengths
\begin{equation}\label{eq:edgelength2}
|e|_{T_i}
  =
  \displaystyle\left\{\begin{array}{ll}
    \frac{\tx(1-\lambda)\norm{A_j}-\lambda\norm{B_j}}{\tx\norm{A_j}}|e|_X
    &
    \text{if } e\in A_j
  \\[1em]
    \frac{\tx\lambda\norm{B_j}-(1-\lambda)\norm{A_j}}{\tx\norm{B_j}}|e|_T
    &
    \text{if } e\in B_j.
  \end{array}\right.
\end{equation}
The length of\/ $\geod$ is
\begin{equation}\label{eq:pathlength2}
  L(\geod) =
  \Big\Arrowvert
  \big(
    \norm{A_1} + \norm{B_1},
    \ldots,
    \norm{A_k}+\norm{B_k} \;
  \big)
  \Big\Arrowvert.
\end{equation}
\end{thm}
\begin{proof}
The presentation in this theorem matches that of the original
Theorem~2.4 in \cite{OwenProvan10} except for the treatment of the
common edges of $X$ and $T$.  Consider any edge $e$ common to $X$ and
$T$.  The definition of a support ensures that $e$ lies in both $A_i$
and $B_i$ for some $i$.  Further, by convention the ratio
$\norm{e}_X/\norm{e}_T$ is negative, so (P2) is never satisfied unless
all of the common edges are placed at the front of the ratio sequence.
This also means that for any $\lambda>0$, each common edge is
contained in some $B_i$ for the computation of its edge length at that
point along the geodesic.  Furthermore, since the common edges are
mutually compatible with each other, they are placed in different
support pairs whenever the ratios $|e|_X/|e|_T$ differ.  It follows
that the common edges are always grouped in support pairs $(A_i,B_i)$
having $\norm{A_i}/\norm{B_i}=- |e|_X/|e|_T$ for any $e$ in that
support pair.

Now consider the length of a common edge $e$ in leg $\geod^i$ of the
path.  By (\ref{eq:edgelength2}),
\begin{align*}
|e|_{T_i}
  &= \displaystyle
      \frac{\tx\lambda\norm{B_j}-(1-\lambda)\norm{A_j}}{\tx\norm{B_j}}|e|_T
      =\left(\lambda - (1-\lambda)\frac{\norm{A_j}}{\norm{B_j}}\right)|e|_T
\\&= \left(\lambda + (1-\lambda)\frac{|e|_X}{|e|_T}\right)|e|_T =
      \lambda|e|_T + (1-\lambda)|e|_X,
\end{align*}
which matches (\ref{eq:common_edges}).

Next look at the term in~(\ref{eq:pathlength2}) corresponding to a
support pair $(A_i,B_i)$ of common edges:
\begin{align*}
\left(\norm{A_j}+\norm{B_j}\right)^2
  &= \left(\frac{\norm{A_j}}{\norm{B_j}}+1\right)^2\norm{B_j}^2
  = \sum_{e\in B_j}\left(1-\frac{|e|_X}{|e|_T}\right)^2|e|^2_T \\
  &= \sum_{e\in B_j}\left(|e|_T-|e|_X\right)^2.
\end{align*}
Summing this over all such pairs $(A_i,B_i)$ yields
$$%
  \sum_{e\in C}\left(|e|_T-|e|_X\right)^2,
$$
where $C$ is the set of common edges.  This matches the expression
given in \cite[Section~4]{OwenProvan10}.

Finally, take the case where an edge $e$ lies in only one of the sets
$X$ and $T$, but is compatible with all edges in the other set.  Intuitively,
we can think of adding $e$ to the other set with length~$0$, and treatin these as
common edges.  Formally, if $e$ lies in~$X$, then it appears in a support
pair $(A_i,\nothing)$ with $A_i$ a set of edges compatible with all of
$T$; and if $e$ lies in~$T$, then it appears in a support pair
$(\nothing,B_i)$ with $B_i$ a set of edges compatible with all of $X$.
Since the ratios of these pairs is either $0$ or~$\infty$,
respectively (since $\norm{\nothing} = 0$), these pairs appear before
and after any nontrivial pairs, respectively.  Further, the edge
component values and path length are as indicated in
(\ref{eq:edgelength2}) and (\ref{eq:pathlength2}), respectively.  This
completes the proof of the theorem.
\end{proof}

\begin{ex}
Figure~\ref{f:treepath} shows an example of the geodesic $\gamma$
between the trees $T$ and $T'$ in Figure~\ref{f:tree_examples} in
Example~\ref{e:tree_examples}.  The associated support $(\cA,\cB)$ for
$\gamma$ has $\cA=\big\{\{e_2,e_3\},\{e_1\}\big\}$ and
$\cB=\big\{\{e_6\},\{e_4,e_5\}\big\}$, and the coordinates of seven
equally spaced trees in $\gamma$ are given in the table.  The length
of this path, as given by~(\ref{eq:pathlength2}), is
$$%
  L(\geod)
  =
  \Big\Arrowvert\big(
  ||\{e_2,e_3\}||+||\{e_6\}||,||\{e_1\}||+||\{e_4,e_5\}||
  \big)\Big\Arrowvert
  =
  15\sqrt 2.
$$

\begin{figure}[ht]
\begin{center}
\includegraphics[width=4.7in]{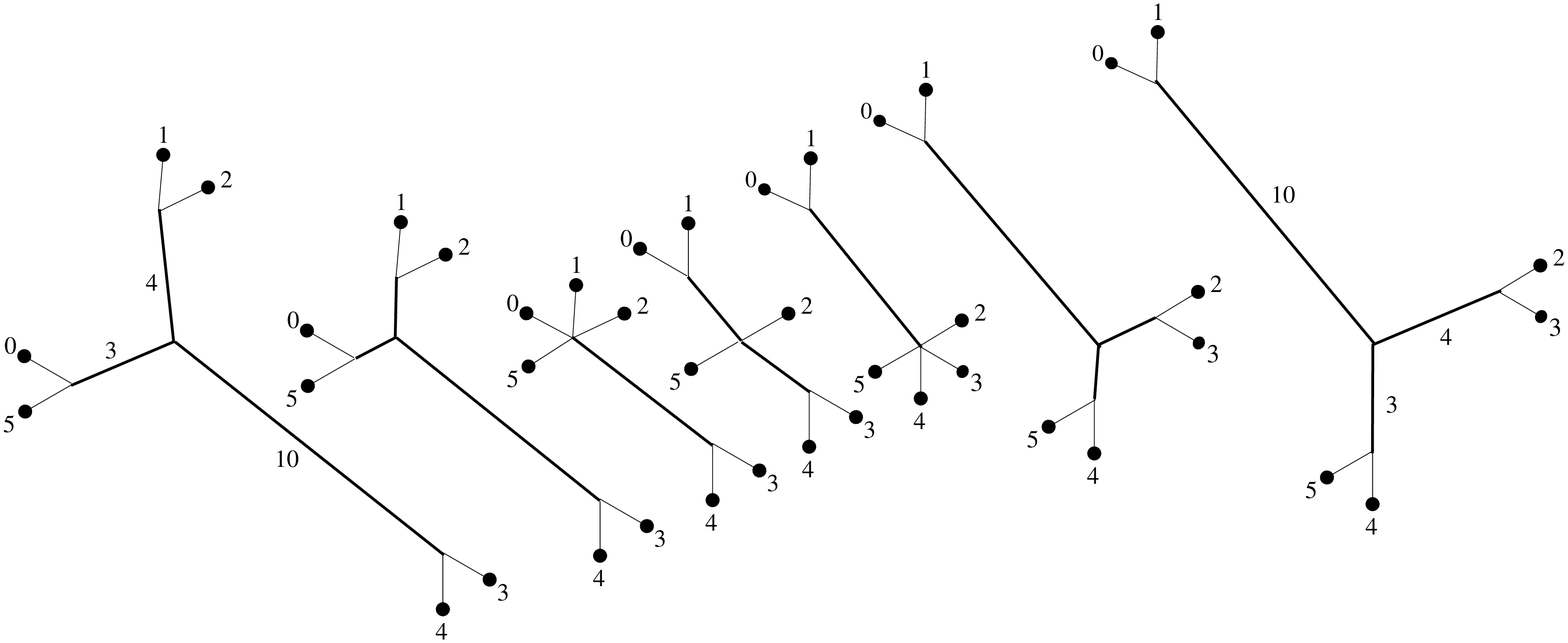}\\[1.5em]
\[\begin{array}{|c|cccccc|}
\multicolumn{1}{c}{}&\multicolumn{6}{c}{\gamma(i/6)}\\
\multicolumn{1}{c}{i}&|e_1|&|e_2|&|e_3|&|e_4|&|e_5|&\multicolumn{1}{c}{|e_6|}\\
\hline
0&10&4&3&0&0&0\\
1&7.5&2&1.5&0&0&0\\
2&5&0&0&0&0&0\\
3&2.5&0&0&0&0&2.5\\
4&0&0&0&0&0&5\\
5&0&0&0&2&1.5&7.5\\
6&0&0&0&4&3&10\\
\hline
\end{array}\]\end{center}
\caption{Seven trees in the geodesic $\gamma$ between $T$ and $T'$,
 sampled at the points $\geod(i/6)$ for $i \in \{0, 1, ..., 6\}$,
The table gives the interior edge lengths for the trees, using the same edge labels as Figure~\ref{f:tree_examples}.}
\label{f:treepath}
\end{figure}
\end{ex}

We end the section by giving a canonical representation for any
geodesic.
\begin{lem}\label{uniquecover}
Any geodesic $\gamma$ can be represented by unique support $(\cA,
\cB)$ satisfying
\begin{equation}\label{eq:strict_inequalities}
\displaystyle
\frac{\norm{A_1}}{\norm{B_1}}
  <
  \frac{\norm{A_2}}{\norm{B_2}}
  <
  \cdots
  <
  \frac{\norm{A_k}}{\norm{B_k}}.
\end{equation}
This support is called the \emph{minimal support}.
\end{lem}
\begin{proof}
This is the content of the remark in \cite[Section~2.3]{OwenProvan10}.
The basic argument is as follows.  Any support $(\cA', \cB')\neq(\cA,
\cB)$ of form (\ref{eq:strict_inequalities}) results in a different
geodesic, since by~(\ref{eq:legs2}) they have different legs.  On the
other hand, for any representation of $\gamma$ having equalities in
the ratio sequence, combine the respective sets in every equality
subsequence.  The resulting support continues to satisfy~(P2), and
hence there is a shortest piecewise linear path from $X$ to~$T$
through the prescribed orthants.  Further, from (\ref{eq:pathlength2})
it follows that the length of this path equals that of~$\gamma$, and
hence defines the unique geodesic~$\gamma$.
\end{proof}

\begin{rem}\label{r:common_edges}
Theorem~\ref{t:prelimThm} positions the support pairs corresponding to
edges compatible with both trees into (\ref{eq:strict_inequalities})
as follows.
\begin{enumerate}[(i)]
\item%
The set $N_X$ of edges of $X$ that are not in $T$ but are compatible
with all edges of $T$ is the set~$A_0$, with $B_0 = \nothing$ and
ratio $\frac{\norm{A_0}}{\norm{B_0}}=\frac{\norm{N_X}}{\norm\nothing}
= -\infty$.
\item%
The set $N_T$ of edges of $T$ that are not in $X$ but are compatible
with all edges of~$X$ is the set $B_k$ with $A_k=\nothing$, and so
its ratio is $\frac{\norm{A_k}}{\norm{B_k}} =
\frac{\norm\nothing}{\norm{N_T}} = 0$.
\item%
Any edge $e$ that lies in both $X$ and $T$ (and hence has positive
length in both sets) appears in both sets of some support pair
$(A_i,B_i)$, and so the ratio is \mbox{$-\infty <
\frac{\norm{A_i}}{\norm{B_i}} < 0$}.
\item%
All other support pairs have $\frac{\norm{A_i}}{\norm{B_i}} > 0$, so
both sets in the support pair are~nonempty.
\end{enumerate}
The ordering of the support pairs in (i)\textendash(iii) has no effect on the
structure of the geodesic between $X$ and~$T$, so for the remainder of
the paper we take the ratio sequence for a geodesic to represent only
the \emph{positive}\/ ratios in the sequence.
\end{rem}

%%%%%%%%%%%%%%%%%%%%%%%%%%%%%%%%%%%%%%%%%%%%%%%%%%%%%%%%%%%%%%%%%%%%%%
\section{The mean and variance in tree space}\label{s:mean}%%%%%%%%%%%
%%%%%%%%%%%%%%%%%%%%%%%%%%%%%%%%%%%%%%%%%%%%%%%%%%%%%%%%%%%%%%%%%%%%%%

Given a finite point set $\bm T = \{T^1,\ldots,T^r\}$ of trees in
$\cT_n$, the \emph{mean} of $\bm T$, alternatively known as the
\emph{Fr\'echet mean}\/ or \emph{barycenter}, is the tree $\ol
T\in\cT_n$ that minimizes the sum $S(X,{\bm T})$ of the squares of the
distances from~$X$ to the points in~$\bm T$.  The \emph{variance} of
$\bm T$ is $S(X,{\bm T})/r$.  Since $r$ is constant throughout the
following discussion, we abuse notation and henceforth refer to the
variance as simply $S(X,{\bm T})$.  The motivation for considering
these notions of mean and variance as the appropriate statistical
objects in tree space was given by Sturm \cite{sturm}, who established
the mathematical foundations for probability theory on global NPC
spaces.  This section reviews the required basics of Sturm's geometric
methods in the context of tree spaces.  (Readers interested in
Fr\'echet means of more general distributions, and those in arbitrary
global NPC spaces, should read Section~\ref{s:npc} now to put the
material below in these more general settings.)

% ; we postpone additional details and
%generality until Section~\ref{s:npc}, which considers more general NPC
%spaces.  In contrast, this section is limited to a precise account of
%mean and variance for trees (Section~\ref{sub:meanTree}) and an
%algorithm to calculate these based on Sturm's method of approximating
%barycenters of probability distributions (Section~\ref{sub:sturmAlg}).

%%%%%%%%%%%%%%%%%%%%%%%%%%%%%%%%%%%%%%%%%%%%%%%%%%%%%%%%%%%%%%%%%%%%%%
\subsection{The variance function}\label{sub:meanTree}%%%%%%%%%%%%%%%%

Let $T \in \cT_n$ be a fixed tree, and consider the geodesics from $T$
to a variable tree $X \in \cT_n$.  The tree $X$ can be thought of as a
vector in~$\RR^\cE_+$, whose coordinates are expressed using the
corresponding lower-case letter~$x$.  If the geodesic from~$X$ to~$T$
has support pair $(\cA,\cB)$ as in Theorem~\ref{t:prelimThm}, then the
squared distance $d(X,T)^2$ from $X$ to $T$ is expressed as the
function
\begin{equation}\label{eq:S}
  S_{T}(x)
  =
  \sum_{i=1}^k\big(\norm{x_{A_i}}+\norm{B_i}\big)^2
  %+\sum_{e\in C}(x_e-|e|_{T})^2.
\end{equation}
in which $x_{A_i}$ is the vector whose coordinates are restricted to
edges in $A_i$.  It follows that for a set ${\bm T}=\{T^1,\ldots,T^r\}
\subseteq \cT_n$ of trees, the variance function $S(X,{\bm T})$ can be
written
\begin{equation}\label{d:var}
  S(x) :=  S(X,{\bm T})= \sum_{\ell=1}^r S_{T^\ell}(x).
\end{equation}
Thus the mean $\ol T$ can be thought of as the point $x^*$ that
minimizes $S(x)$ over $x \in\cT_n$.

To state the next result, a real-valued function $f: \cT \to \RR$ on a
metric space~$\cT$ is \emph{strictly convex} if~$f \circ \gamma$ is a
strictly convex real-valued function on~$\RR$ for all
geodesics~$\gamma$; that is, if
$$%
  f\big(\gamma(\lambda)\big)
  <
  (1-\lambda)f\big(\gamma(0)\big) + \lambda f\big(\gamma(1)\big)
  \text{ whenever } 0 < \lambda < 1.
$$

\begin{prop}\label{p:convex}
The variance function $S(x)$ is strictly convex as a function
on~$\cT_n$.  Consequently, the mean is the unique local minimum
of~$S(x)$ in~$\cT_n$.
\end{prop}
\begin{proof}
\cite[Proposition 1.7]{sturm}.  See also Example~\ref{e:dconvex}.
\end{proof}

The differentiability of the variance function $S$ is critical to the
construction of gradient-descent methods for minimizing $S$.  This in
turn depends on the differentiability of the individual functions
$S_T$ in (\ref{d:var}).  Identifying the geodesic from $X$ to~$T$ by
its support $(\cA,\cB)$, Eq.~\eqref{eq:S} yields the partial
derivatives of $S_T$ with respect to each of the coordinates~$x_e$:
\begin{align}\label{eq:partialS}\nonumber
\frac{\partial S_T(X)}{\partial x_e}
&= 2\big(\norm{x_{A_i}} + \norm{B_i}\big) \frac{x_e}{\norm{x_{A_i}}}
\\
&= 2x_e\Big(1 + \frac{\norm{B_i}}{\norm{x_{A_i}}}\Big),
\end{align}
where $A_i$ is the set containing $e$.  This is well-defined whenever
$x$ lies in the interior of its maximal orthant, although the
functional form of (\ref{eq:partialS}) depends upon the combinatorial
type of the geodesic, and in particular on $(\cA,\cB)$.  It turns out,
however, that throughout the interior of any maximal orthant the
function $S$ is continuously differentiable.

\begin{thm}\label{t:diff}
The variance function $S(x)$ is continuously differentiable on the
interior of every maximal orthant~$\Or$.
\end{thm}
\begin{proof}%
{}From (\ref{d:var}) it suffices to show that the
function~(\ref{eq:partialS}) is continuously differentiable on the
interior of~$\Or$.  By Lemma~\ref{uniquecover} the geodesic between
$X$ and $T$ can be represented uniquely by support $(\cA^0,\cB^0)$
satisfying (\ref{eq:strict_inequalities}).  Any other support
$(\cA,\cB)$ for this geodesic consists of a sequence of sets
partitioning the $A^0_i$ and $B^0_i$ into equality subsequences, with
the ratios $\norm{A_j}/\norm{B_j}$ in the equality subsequence equal
to the corresponding ratio $\norm{A^0_i}/\norm{B^0_i}$ of the sets
from which they were partitioned.  But this means that the ratio
$\norm{B_i}/\norm{x_{A_i}}$, and hence $\partial S_T(X)/\partial x_e$,
is the same regardless of which representation we choose for the
geodesic.  It follows that the partial derivatives are continuous
everywhere in the interior of~$\Or$.
\end{proof}

%%%%%%%%%%%%%%%%%%%%%%%%%%%%%%%%%%%%%%%%%%%%%%%%%%%%%%%%%%%%%%%%%%%%%%
\subsection{Sturm's algorithm}\label{sub:sturmAlg}%%%%%%%%%%%%%%%%%%%%

The orthant structure of tree space~$\cT_n$ prevents the averaging of
finite point sets using the standard Euclidean centroid.  The
following serves as an approximate replacement, introduced by Sturm
\cite[Definition~4.6]{sturm} in the context of probability theory on
arbitrary globally nonpositively curved spaces.

\begin{defn}\label{d:indMean}
For a set $X^1,X^2,\ldots$ of points in $\cT_n$ and an index~$k$, the
\emph{inductive mean value} of $X^1,\ldots,X^k$ is the point $\mu_k$
%$$%
%  \mu_k
%  =
%  \displaystyle\frac1k\sum_{\ell=1,\ldots,k}^{\longrightarrow} X^\ell
%$$
defined by setting $\mu_1 = X^1$ and for $\ell=2,\ldots,k$, letting
$\mu_\ell$ be the point $\gamma_{1/{\ell}}=\gamma(1/{\ell})$ that is
$1/{\ell}$ along the geodesic $\gamma$ from $\mu_{{\ell}-1}$ to
$\gamma_1 = X^{\ell}$.
\end{defn}

Note that if all of the $X^{\ell}$ lie in the same orthant, then the
inductive mean value of $X^1,\ldots,X^k$ in fact equals the standard
centroid of $X^1,\ldots,X^k$.  For general points in $\cT_n$, though,
the inductive mean may not be the Fr\'echet mean; it may in fact give
different points for different orderings of the
points~$X^1,\ldots,X^k$ (see Example~\ref{e:indMean}).  Sturm goes on
to prove \cite[Theorem~4.7]{sturm} the following \emph{strong law of
large numbers} for the Fr\'echet mean.

\begin{thm}\label{t:sturmPoints}
Fix a set $\{T^1,\ldots,T^r\} \subseteq \cT_n$ of trees.  If
$X^1,X^2,\ldots$ is a sequence of points sampled uniformly and
independently from $\{T^1,\ldots,T^r\}$, then with probability~$1$,
the sequence of inductive mean values $\mu_1,\mu_2,\ldots$ approaches
the mean $\ol T$ of $\{T^1,\ldots,T^r\}$.
%; that is,
%$$%
%  \frac1k\sum_{\ell=1,\ldots,k}^{\longrightarrow}X^\ell \to \ol T.
%$$
\end{thm}

To be precise, Sturm shows that if we take the inductive mean $\mu_k$
as a random variable dependent on the sampling of the points~$X^\ell$,
then the distance $d(\mu_k,\ol T)$ from~$\mu_k$ to the true mean~$\ol
T$ has expected value bounded above by $S({\bm T},\ol T)/k$.  This
gives us a way of estimating the Fr\'echet mean $\ol T$ through a
sequence of inductive means $\mu_1,\mu_2,\ldots$ obtained by randomly
sampling trees from the set $\{T^1,\ldots,T^r\}$.

\vbox{
\begin{alg}[Sturm's algorithm]\label{a:sturm}
\end{alg}
\begin{alglist}
\begin{routinelist}{input}
    \item[a set $\{T^1,\ldots,T^r\}$ of trees in $\cT_n$]
    \item[positive integers $K$ and $N$]
    \item[positive real number $\varepsilon$ ]
\end{routinelist}
\routine{output} $\mu_k = k^{\text{th}}$ approximation of the mean tree
\begin{routinelist}{initialize}
\item [choose a tree $T \in \{T^1,\ldots,T^r\}$ uniformly at random]
\item [set $\mu_1 := T$]
\item [set $k := 1$]
\end{routinelist}
\routine{while} $k<K$ or pairwise distances $d(\mu_j, \mu_\ell)$ for $k - N <
j,\ell \leq k$ are not all $\leq \varepsilon$
\begin{routinelist}{do}
    \item [choose tree $T \in \{T^1,\ldots,T^r\}$ uniformly at random]
    \item [set $\gamma$ := the geodesic from $T$ to $\mu_k$]
    \item [set $\mu_{k+1} := \gamma_{1/(k + 1)}$]
    \item [set $k := k + 1$]
\end{routinelist}
\routine{end}{}\procedure{while-do}
\routine{return} $\mu_k$, the $k^{th}$ approximation of the mean tree
\end{alglist}
}

\begin{rem}\label{r:precision}
The choice of stopping criterion involves two parts.
\begin{enumerate}[(i)]
\item%
Running the algorithm a specified initial number $K$ of iterations
guarantees an upper bound of $\frac{r}{K+1}S(\ol T,{\bm T})$ on the
expected distance of the final tree $\mu_k$ to the mean~$\ol T$.  This
is derived from the proof of Theorem 4.7 in \cite{sturm}.
\item%
Comparing the final $N$ sample means
serves as a proxy for testing that the sample means
act like a Cauchy sequence for $N$ steps.
\end{enumerate}
Thus in principle, proper settings for $K$, $N$, and $\varepsilon$
could be used to set confidence intervals on the distance $d(\mu_i,\ol
T)$ by using Sturm's result.  This would involve a more sophisticated
statistical analysis, which we did not undertake in this paper.  
In practice, we chose $K$, $N$, and $\varepsilon$ to balance
run-time with the desired precision, working under the rough
assumption that if $N$ is chosen
large enough, then the approximate mean will be within $\varepsilon$ of 
the mean tree.  For Example~\ref{e:gene}, we chose $K = 1\,000\,000$, $N = 10$, and 
$\varepsilon=10^{-4}$.
%For
%our experiments, such as those reported in 
%%Examples~\ref{e:brain} and~
%\ref{e:gene}, we chose $K = 1\,000\,000$ and $N = 5$.  We either
%fix~$\varepsilon$ to be some value, as in Example~\ref{e:brain}, or we
%calculate it based on the square root of the sample variance of the
%sample mean tree $\mu_{5r}$.  In Example~\ref{e:gene}, for instance,
%we use $\varepsilon=10^{-5}(S(\mu_{5r},{\bm T}))^{1/2}$.  If a
%function is utilized to determine any of these parameters, then it
%should depend on the size~$r$ of the set of trees to be averaged.
\end{rem}

\begin{rem}\label{r:implement}
We have made software implementing this algorithm freely available
\cite{Owen12}.
\end{rem}

%%%%%%%%%%%%%%%%%%%%%%%%%%%%%%%%%%%%%%%%%%%%%%%%%%%%%%%%%%%%%%%%%%%%%%
\section{The combinatorics of geodesics in tree space}\label{s:vistal}
%%%%%%%%%%%%%%%%%%%%%%%%%%%%%%%%%%%%%%%%%%%%%%%%%%%%%%%%%%%%%%%%%%%%%%

This section investigates the combinatorial structure of geodesics
in~$\cT_n$ and their relationship to the variance function.  To be
more precise, fix a source %
%
%\footnote{The terminology comes from studies of polyhedral unfolding
%\cite{unfolding}, where the unmoving point $T$ emits a signal at unit
%speed, so at time $d(X,T)$ the wavefront passes through~$X$.}
%
tree $T \in \cT_n$.  The shortest path from an arbitrary tree $X \in
\cT_n$ to~$T$ has a ``combinatorial type'', determined through
Theorem~\ref{t:prelimThm} by the sequence of orthants that it passes
through, or more specifically the support pair $(\cA,\cB)$ associated
with the geodesic.  This combinatorial type can change, even when $X$
has the same topology, depending on the precise values of the lengths
of the edges in $X$.  We are interested in the partition of
$\cT_n$ \textemdash{} called%
\footnote{Our use of the term ``vistal subdivision'' here differs from
\cite[Conjecture~9.6]{unfolding}: vistal facets in
Definition~\ref{d:vistal} here are analogous to \emph{cut cells} in
\cite[Definition~5.4]{unfolding}.  In contrast, the equivalence
relation in
% \cite[Conjecture~9.6]{unfolding}
\cite{unfolding} declares two points \emph{equivistal} when their
vistal subdivisions---in the sense of
Theorem~\ref{t:cell_complex}---are combinatorially the same.}
the \emph{vistal subdivision} of $\cT_n$---into regions for which the
geodesics to the fixed tree $T$ have the same combinatorial type.

We begin by describing a simple change of coordinates, the
\emph{squaring map}, and characterizing the faces of maximal dimension
in the vistal subdivision (Section~\ref{sub:vistal}).  In particular,
Propositions~\ref{p:cone} and \ref{p:polyhedral} establish that after
applying the squaring map the vistal facets are polyhedral regions
that cover tree space but have disjoint interiors.  Next we provide a
simple description of the faces of lower dimension in these polyhedra
(Section~\ref{sub:lowerdim}).  Finally, we prove that the vistal
facets constitute the maximal cells of a polyhedral subdivision of
tree space, called the \emph{vistal polyhedral subdivision}
(Section~\ref{sub:vistalSub}).

\begin{rem}
The idea of studying the paths taken by geodesics emanating from a 
source point has been studied in computational geometry, in the areas
of single source shortest path queries \cite{Mitchell00} and polyhedral unfolding
\cite{unfolding}.  Recently Chepoi and Maftuleac studied the single source
 shortest path problem for CAT(0) rectangular
complexes, where each cell is a 2D rectangle.  When the underlying 
space is intrinsically 2D, these shortest path subdivisions are often polyhedral.
However, in general, we can not expect this in higher dimensions, and indeed, 
without the squaring map, the vistal subdivisions are not polyhedral,
as illustrated in Example~\ref{e:subdivision}.  The squaring map is only possible
in $\cT_n$ because all of the combinatorial complexity of the space happens
about the origin. 
\end{rem}

\begin{ex}\label{e:4dim}
Vistal subdivisions are in general far from polyhedral, even after
changes of coordinates such as squaring.  A prerequisite for a
squaring map to produce polyhedrality would be that (every component
of the) the bounding hypersurface has degree at most~$2$.  However,
global NPC cubical complexes can have vistal cells bounded by
hypersurfaces of degree greater than~$2$.  We conjecture that the
bounding hypersurface can have components of arbitrarily high degree.

For a specific example, consider first the arrangement
$$%
  \includegraphics[height=30ex]{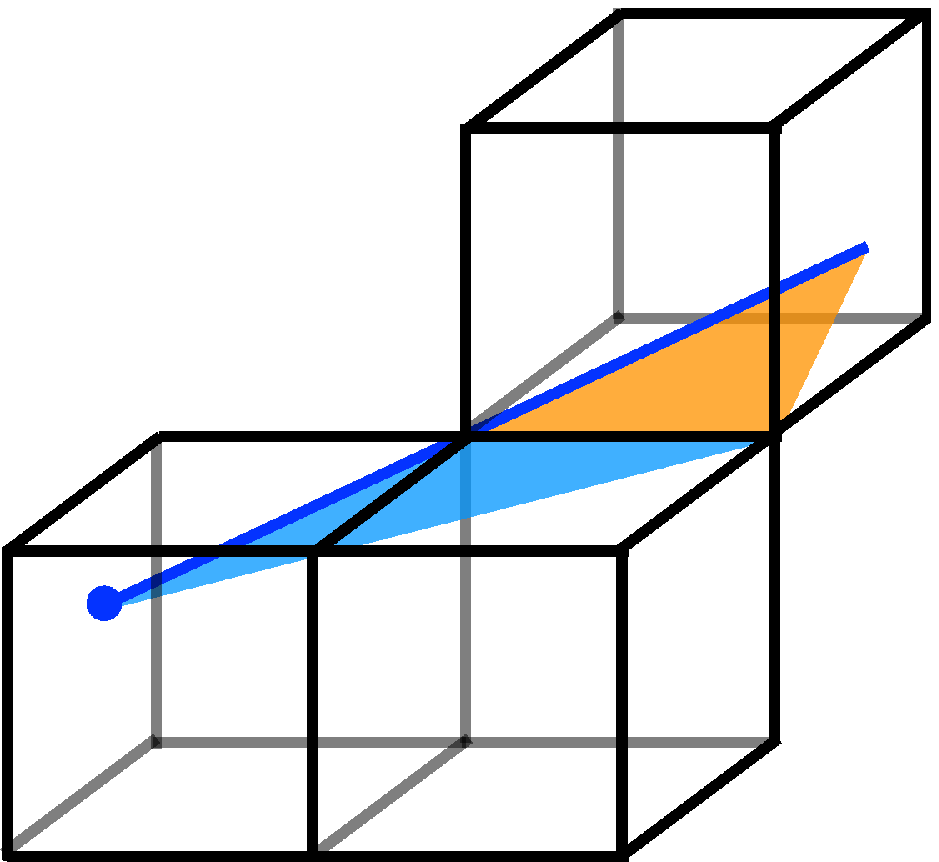}
$$
of three $3$-dimensional cubes in which the bottom two cubes are
joined along a common face, the top cube meets one bottom cube along
an edge, and all three cubes meet at a central vertex.  Consider the
set~$\cV$ of points whose shortest path back to a fixed starting point
(the dot in the left bottom cube) passes through the interior of the edge joining the
top cube to the bottom cubes, as opposed to passing through the
central vertex.  The set~$\cV$ is obtained by rotating the top shaded
triangle about the shared cube edge.  The boundary of~$\cV$ is a cone
and hence is described as the vanishing locus (in the top cube) of a
polynomial of degree~$2$.

To get the desired example, glue a $4$-cube to the right-hand $2$-face
of the top cube in the figure.  If it were merely a $3$-cube glued on,
then the cone~$\cV$ would simply continue to expand into the added
cube, creating a frustum in the added $3$-cube.  But once the new
facet is a $4$-cube, the frustum rotates freely around the shared
$2$-face.  The boundary of such a rotated frustum is the locus of
zeros (in the $4$-cube) of a polynomial, but any such polynomial~$f$
has degree at least~$4$.  Indeed, intersecting the hypersurface in
question with a generic hyperplane yields a union of two cones,
% (any 2-plane that cuts the boundary of the frustum into two lines
% rotates around the 2-face to yield two cones)
each of which has degree~$2$.  Therefore $f$ restricts to a polynomial
of degree at least~$4$ on the hyperplane, whence the hypersurface
itself has degree at least~$4$.
\end{ex}

%%%%%%%%%%%%%%%%%%%%%%%%%%%%%%%%%%%%%%%%%%%%%%%%%%%%%%%%%%%%%%%%%%%%%%
\subsection{Vistal facets}\label{sub:vistal}%%%%%%%%%%%%%%%%%%%%%%%%%%

\begin{defn}\label{d:vistal}
Given a source tree $T \in \cT_n$, a maximal orthant $\Or \subseteq
\cT_n$, and a support $(\cA, \cB)$, let $\cV(T,\Or; \cA, \cB)$ be the
closure of the set of trees $X \in \Or$ for which the geodesic joining
$X$ to~$T$ has support $(\cA,\cB)$ satisfying (P2) and (P3) with
strict inequalities.  A \emph{previstal facet}\/ is any nonempty set
$\cV(T,\Or; \cA, \cB)$ of this form.
\end{defn}

The description of $\cV(T,\Or; \cA,\cB)$ becomes linear after a simple
change of variables.

\begin{defn}\label{d:squaring}
The \emph{squaring map} $\cT_n \to \cT_n$ acts on $x \in \cT_n
\subseteq \RR^E_+$ by squaring coordinates:
$$%
  (x_e \mid e \in E) \mapsto (\xi_e \mid e \in E),
  \text{ where }
  \xi_e = x_e^2.
$$
Denote by $\cT_n^2$ the image of this map, and let $\xi_e=x^2_e$
denote the coordinate indexed by $e \in E$.  The image of an orthant
in $\cT_n$ is then the equivalent orthant in $\cT_n^2$, and the image
of a previstal facet $\cV(T,\Or; \cA, \cB)$ in~$\cT^2_n$ is a
\emph{vistal facet} denoted $\cV^2(T,\Or; \cA, \cB)$.  With this
change of variables, $\norm A = \sum_{e\in A} \xi_e$ for any set of splits $A$.
\end{defn}

The squaring map induces on the variance function $S$ a corresponding
pullback function
\begin{equation}\label{e:squared_variance}
  S^2(\xi)
  =
  S(\sqrt\xi\,), \text{ where } (\sqrt\xi\,)_e = \sqrt{\xi_e}\,.
\end{equation}
Since the variance function $S(x)$ is continuous on~$\cT_n$ with a
uniquely attained minimum by Proposition~\ref{p:convex}, and
continuously differentiable on the interior of each maximal orthant by
Theorem~\ref{t:diff}, the same properties hold for $S^2$.  Thus we can
apply steepest descent methods
% in the squared coordinates as we would in the original ones.
after squaring just as we would beforehand.  This is further explored
in~Section~\ref{s:computing}.

Theorem~\ref{t:prelimThm} implies a nice description of the vistal
facets of~$\cT^2_n$.

\begin{prop}\label{p:cone}
The vistal facet $\cV^2(T,\Or; \cA, \cB)$ is a convex polyhedral cone
in $\cT^2_n$ defined by the following inequalities on $\xi \in
\RR^{E\!}$, where all norms $\norm{\cdot}$ are to be
interpreted~as~$\norm{\cdot}_{T}$.
\begin{enumerate}[\quad\rm (P1)]\setlength\itemsep{-.5ex}
\item[{\rm (O)}]\addtocounter{enumi}{1}%
$\xi \in \Or$; that is, $\xi_e \geq 0$ for all $e \in E$, and $\xi_e =
0$ for $e \notin \cE$, where $\Or = \RR^\cE_+$.

\item%(P2)
$\displaystyle\norm{B_{i+1}}^2 \sum_{e \in A_i} \xi_e \leq
\norm{B_i}^2\!\! \sum_{e \in A_{i+1}} \xi_e$ for all $i=1,\ldots,k-1$.

\item%(P3)
$\displaystyle\norm{B_i\!\minus\! J}^2\!\!\! \sum_{e \in A_i \minus I}
\xi_e \geq \norm{J}^2 \sum_{e \in I} \xi_e$ for all $i=1,\ldots,k$
and subsets $I \subseteq A_i$, $J\subseteq B_i$ such that $I \cup J$
is compatible.
\end{enumerate}
\end{prop}
\begin{proof}
For a vector $x = (x_e \mid e \in E) \in \Or$ to lie in $\cV(T,\Or;
\cA, \cB)$, the tree $X$ must be in an orthant of $\cT_n$ and satisfy
properties (P2) and~(P3).  The orthant condition immediately implies
the nonnegativity conditions in (O).  The inequalities corresponding
to (P2) are
$$%
  \frac{\norm{A_i}}{\norm{B_i}}
  \leq
  \frac{\norm{A_{i+1}}}{\norm{B_{i+1}}}.
$$
Squaring, cross-multiplying, and substituting~$\xi_e$ for $x^2_e$
yields the corresponding linear inequality in $\cV^2(T,\Or; \cA, \cB)$
in~$\cT^2_n$.  The inequalities for (P3) are obtained in the same
manner.
\end{proof}

\begin{prop}\label{p:polyhedral}
The vistal facets are of dimension $2n-1$, have pairwise disjoint
interiors, and cover $\cT^2_n$.  A point $\xi \in \cT_n^2$ lies
interior to a vistal facet $\cV^2(T,\Or; \cA, \cB)$ if and only if the
inequalities in (O), (P2), and~(P3) are strict.
\end{prop}
\begin{proof}
The vistal facets cover $\cT^2_n$ by definition: $T \in \cV^2(T,\Or;
\cA, \cB)$ if the geodesic from $X \in \Or$ to~$T$ has support $(\cA,
\cB)$.  The second statement follows by definition~and by standard
properties of convex polyhedra presented as solutions to systems of
linear inequalities.
\end{proof}

%%%%%%%%%%%%%%%%%%%%%%%%%%%%%%%%%%%%%%%%%%%%%%%%%%%%%%%%%%%%%%%%%%%%%%
\subsection{Vistal cells}\label{sub:lowerdim}%%%%%%%%%%%%%%%%%%%%%%%%%

Henceforth in this section we focus our attention on the squared tree
space $\cT^2_n$ and its expression as a union of polyhedral vistal
facets as given by Proposition~\ref{p:cone}.
% \begin{defn}\label{d:vistalCell}
% A \emph{vistal cell}\/ is any face of a vistal facet; see
% Definition~\ref{d:signature} for more detail.
% \end{defn}
% \noindent
This subsection concerns faces of vistal facets, including compact
characterizations thereof.
% (Definition~\ref{d:signature});
% and their arrangements in the vistal subdivision.
% (It turns out that every face of a vistal facets is a vistal face;
% see Theorem~\ref{t:cell_complex}, particularly the verification of
% condition~(C1) in its proof).

\subsubsection{Signatures and vistal cells}\label{ss:ratio}%%%%%%%%

\begin{defn}\label{d:signature}
Fix a source tree $T \in \cT_n$, a (not necessarily maximal) orthant
$\Or \subseteq \cT_n$, and a support $(\cA, \cB)$.  A
\emph{signature}\/ associated with the support $(\cA, \cB)$ is a
length $k-1$ sequence $\cS = (\sigma_1,\ldots, \sigma_{k-1})$ of
symbols $\sigma_i \in \{=\:,\,\leq\}$.  The
% \emph{(closed) previstal cell}\/
\emph{previstal cell}\/ defined by $\Or$, $\cA$, $\cB$, and $\cS$ is
the set $\cV(T, \Or; \cA, \cB; \cS)$ of points $X$ in $\Or$ for
which the ratio sequence for $(\cA,\cB)$ at the point~$X$ has the
following specific form:
\begin{equation}\label{eq:ratio_sequence}
\displaystyle
  \frac{\norm{A_1}}{\norm{B_1}}
  \sigma_1
  \frac{\norm{A_2}}{\norm{B_2}}
  \sigma_2
  \cdots
  \sigma_{k-2}
  \frac{\norm{A_{k-1}}}{\norm{B_{k-1}}}
  \sigma_{k-1}
  \frac{\norm{A_k}}{\norm{B_k}}.
\end{equation}
The
% \emph{(closed) vistal cell}
\emph{vistal cell} $\cV^2(T, \Or; \cA, \cB; \cS) \subseteq \cT_n^2$ is
the image of $\cV(T, \Or; \cA, \cB; \cS)$ under~squaring.
\end{defn}

\begin{rem}\label{r:cell}
Vistal cells are convex polyhedra that need not be bounded,
% see Theorem~\ref{t:vistal_face_combinatorics}
and as such they might not be topological cells.  However, the
interior of a convex polyhedron is a topological cell, so every vistal
cell is the closure of a topological cell.
\end{rem}

\begin{lem}\label{l:full_dimensional}
The dimension of the vistal cell $\cV^2(T, \Or; \cA, \cB; \cS)$ is at
most $\dim(\Or) - m(\cS)$, where $m(\cS)$ is the number of~``=''
components in~$\cS$.  The vistal cell $\cV^2(T, \Or; \cA, \cB; \cS)$
is full-dimensional if and only if there exists a point $X\in \cV(T,
\Or; \cA, \cB; \cS)$ satisfying the following two properties.
\begin{enumerate}[\quad\rm(V1)]
\item%(V1)
For each $i=1,\ldots,k-1$, $\frac{\norm{A_i}}{\norm{B_i}}=
\frac{\norm{A_{i+1}}}{\norm{B_{i+1}}}$ if $\sigma_i$ is~``='' and
$\frac{\norm{A_i}}{\norm{B_i}}< \frac{\norm{A_{i+1}}}{\norm{B_{i+1}}}$
if $\sigma_i$ is~``$\leq$''.
\item%(V2)
The inequalities in (P3) are satisfied strictly.
\end{enumerate}
\end{lem}
\begin{proof}
This follows from standard polyhedral theory, as treated in
\cite{ziegler}, for instance.
\end{proof}

Proposition~\ref{p:cone} implies that (i)~vistal cells are faces of
vistal facets, and that (ii)~vistal facets are vistal cells for which
$\Or$ is maximal and the signature contains only ``$\leq$'' symbols.
What we prove here is that \emph{all}\/ faces of vistal facets can be
represented as vistal cells, and that under some simple conditions on
$(\cA, \cB)$, Definition~\ref{d:signature} provides a canonical
description of each vistal cell.  We start by determining all supports
and signatures associated with the geodesic $\gamma$ from $T$ to a
particular point $X$.  By Lemma~\ref{uniquecover}, the
geodesic~$\gamma$ can be represented by a unique minimal support
$(\cA, \cB)$ satisfying (\ref{eq:strict_inequalities}):
$$%
\displaystyle
\frac{\norm{A_1}}{\norm{B_1}}
  <
  \frac{\norm{A_2}}{\norm{B_2}}
  <
  \cdots
  <
  \frac{\norm{A_k}}{\norm{B_k}}.
$$
Any other support $(\cA', \cB')$ of $\gamma$ corresponds to a ratio
sequence in which at least one ratio $\norm{A_i}/\norm{B_i}$ is
replaced by a \emph{ratio subsequence} formed from a partition of
$A_i$ and $B_i$, with equalities between all terms.  Any ratio
subsequence for which $X$ continues to satisfy (P3) together with
equalities between terms of the ratio subsequences constitutes a valid
support for $\gamma$.  We next give a specific method for determining
all such support sequences.

\subsubsection{Incompatibility graphs and equality subsequences}%%%%%%
\label{ss:incompatibility}%%%%%%%%%%%%%%%%%%%%%%%%%%%%%%%%%%%%%%%%%%%%

In \cite{OwenProvan10} it was shown how condition~(P3) for
support pair $(A_i, B_i)$ can be rephrased in terms of conditions on a
special node-weighted graph derived from the compatibility relations
between $X$ and $T$ and their coordinate values.  We summarize the
technique here.  Denote the coordinates of $X$ and $T$ by $X = (x_e
\mid e \in \cE_X)$ and $T = (t_e \mid e \in \cE_T)$, and let $\xi_e =
x_e^2$ and $\tau_e = t_e^2$ be their squared coordinates.

\begin{defn}\label{d:incompatibilityGraph}
The \emph{incompatibility graph}\/ $G(A_i,B_i)$ between $A_i$ and
$B_i$ is the weighted bipartite graph with vertex set $A_i\cup B_i$
and an edge from $a\in A_i$ to $b\in B_i$ whenever $a$ and~$b$ are
incompatible.  The weight of each vertex $a\in X$ is $\tilde{\xi}_a =
\xi_a/\sum_{e\in A_i}\xi_e$, and the weight of each vertex $b\in T$ is
$\tilde{\tau}_b = \tau_b/\sum_{e\in B_i}\tau_e$.  A \emph{(vertex)
cover} for $G(A_i,B_i)$ is a set $C\subset A_i\cup B_i$ having the
property that every edge of $G(A_i,B_i)$ has at least one endpoint in
$C$.  The weight of $C$ is the sum of the weights of its vertices.
\end{defn}

\begin{lem}[{\cite[Section~3]{OwenProvan10}}]\label{l:P3_for_incomp_graph}
Property (P3) holds for support pair $(A_i, B_i)$ if and only if every
cover of $G(A_i, B_i)$ has weight~$\geq 1$.\qed
\end{lem}

By Lemma~\ref{l:P3_for_incomp_graph}, testing a support pair $(A_i,
B_i)$ for property (P3) is equivalent to showing that the \emph{min
weight cover}\/ for $G(A_i, B_i)$ has weight 1. The problem of finding
the minimum cover in $G(A_i, B_i)$ in turn can be reduced to solving a
max flow problem (see \cite{AhujaEtAl93}, Section 12.3) on a specially
defined flow network $F(A_i,B_i)$.  To construct $F(A_i,B_i)$, start
with $G(A_i,B_i)$, attach a source~$\bar s$ to the $A_i$-vertices of
$G(A_i,B_i)$ and a sink~$\bar t$ to the $B_i$-vertices
of~$G(A_i,B_i)$, and direct all edges from $\bar s$ toward $\bar t$.
Set the capacity of each edge $(\bar s,a)$ to~$\tilde{\xi}_a$, set the
capacity of each edge $(b,\bar t)$ to~$\tilde{\tau}_b$, and set the
capacities of edges in $G(A_i,B_i)$ to~$\infty$.  The Max-Flow-Min-Cut
Theorem implies that the value of the maximum $(\bar s,\bar t)$-flow
$f$ for $F(A_i,B_i)$ is equal to the capacity of a minimum capacity of
an $(\bar s,\bar t)$-cut $K$ in $F(A_i,B_i)$, which in turn
corresponds to a minimum weight cover $C$ for $G(A_i,B_i)$.  Thus the
condition in Lemma~\ref{l:P3_for_incomp_graph} for $G(A_i,B_i)$ is
equivalent to the property that the max flow in $F(A_i,B_i)$ is $\geq
1$.  The precise relationship between max flows in $F(A_i,B_i)$ and
min covers in $G(A_i,B_i)$ is crucial to determining the possible
ratio subsequences that can replace a term $\norm{A_i}/\norm{B_i}$ in
(\ref{eq:strict_inequalities}), and we clarify this relationship
below.

\begin{ex}\label{e:picard_queyranne}
Figure~\ref{f:picard_queyranne}
\begin{figure}
%\centering
\scalebox{0.8}{ \input{picard_queyranne.pstex_t} }
%\vspace{3.5in}\comment{[...page-long figure goes here...]}\vspace{3.5in}
\caption{Characterizing ratio subsequences}
\label{f:picard_queyranne}
%\centering
\end{figure}
demonstrates this for a hypothetical support pair $(A_i,B_i)$ with
$A_i = \{x_1,x_2,x_3,x_4,x_5,x_6,x_7,x_8\}$ and $B_i =
\{t_1,t_2,t_3,t_4,t_5,t_6,t_7\}$, compatibility graph $G(A_i,B_i)$,
and values $\xi_a$, and $\tau_b$ as given in
Figure~\ref{f:picard_queyranne}(a).
Figure~\ref{f:picard_queyranne}(b) depicts the associated flow graph
$F(A_i,B_i)$ and max flow.  For simplicity, the weights are not
normalized, so that all numbers are scaled by $23$, the sum of the
weights.  This flow has value 23, which means that the pair
$(A_i,B_i)$ satisfies (P3).
\end{ex}

\subsubsection{Residual graphs and ratio subsequences}\label{ss:residual}

Now consider the problem of determining the possible ratio
subsequences replacing a term $\norm{A_i}/\norm{B_i}$ in the ratio
sequence of a minimal support for $X$ and $T$.  We use the optimal
flow conditions on $F(A_i,B_i)$ to do this.  Recall that here $(A_i,
B_i)$ also satisfies (P3), so that the max flow $f$ on $F(A_i,B_i)$
has value~$1$.  The associated minimum weight cover for $G(A_i,B_i)$
can then be obtained from this flow.  To do this, we define another
auxiliary graph.

\begin{defn}\label{d:residual}
The \emph{residual graph} $G_i^r$ with respect to $f$ has
\begin{enumerate}[\quad(a)]\setlength\itemsep{-.5ex}
\item%
all edges of $G(A_i,B_i)$, directed as in $F(A_i,B_i)$, and
%all edges of $G(A_i,B_i)$;
\item%
all edges~$e$ of $F(A_i,B_i)$ \textemdash{} but in the reverse direction \textemdash{} where
$f_e > 0$.
%an edge~$e$ of $G(A_i,B_i)$---but in the reverse direction---if $f_e >
%0$; and
%\item%
%any edge adjacent to the source~$\bar s$ or sink~$\bar t$---but in the
%reverse direction.
\end{enumerate}
An \emph{$(\bar s,\bar t)$-cut}\/ in $G_i^r$ is any partition
$(H,\oH)$ of the nodes of $G_i^r$ having the property that no edge
of~$G_i^r$ goes from~$H$ to~$\oH$.
\end{defn}

It is easy to see that by this definition, $H$ contains $\bar s$ and
$\oH$ contains $\bar t$.
The definition of residual graph is based on the structure of
$F(A_i,B_i)$ and the fact that the flow $f$ saturates (is at capacity
on) all edges adjacent to either $\bar s$ or $\bar t$.  The
Max-Flow-Min-Cut Theorem states that every $(\bar s,\bar t)$-cut in
$G_i^r$ corresponds to a cut of capacity 1 in $F(A_i,B_i)$, which in
turn corresponds to a cover of weight 1 in $G(A_i,B_i)$.  This leads
to the following result.

\begin{lem}\label{l:residual}
Let $(H,\oH)$ be a $(\bar s,\bar t)$-cut in the residual graph $G_i^r$.  Then the sets
$I_1= \oH\cap A_i$, $J_1=\oH\cap B_i$, $I_2 = H\cap A_i$, and $J_2=
H\cap B_i$ have the property that $\frac{\norm{I_1}}{\norm{J_1}} =
\frac{\norm{I_2}}{\norm{J_2}}$ can replace
$\frac{\norm{A_i}}{\norm{B_i}}$ in (\ref{eq:strict_inequalities}) and
the resulting sequence still satisfies (P2) and (P3).
\end{lem}
\begin{proof}
By Definition \ref{d:residual}(a) all edges of $G(A_i,B_i)$ are in
$G^r_i$, so in particular there can be no edge from any element in
$I_2$ to any element in $J_1$.  Thus $I_2\cup J_1$ is compatible.
Further, by Definition \ref{d:residual}(b) there are no edges of
$G^r_i$ from $H$ to $\oH$, so the flow is conserved in~$H$, and hence
in~$\oH$.  This implies $\norm{I_1} = \norm{J_1}$ and $\norm{I_2} =
\norm{J_2}$, and thus the ratios are equal.  Finally, since the flow
$f$ restricted to each of the subgraphs $F(I_1,J_1)$ and $F(I_2,J_2)$
continues to saturate the edges adjacent to $\bar s$ and~$\bar t$,
property (P3) continues to be satisfied on the replacement support
pairs $(I_1,J_1)$ and~$(I_2,J_2)$.
\end{proof}

\begin{ex}[continuation of Example~\ref{e:picard_queyranne}]
One min cut with respect to the flow in
Figure~\ref{f:picard_queyranne}(b) has $\oH = \{x_1,x_2,x_3,t_1,\bar
t\}$ and~$H$ its complement; this corresponds to the pairs $I_1 =
\{x_1,x_2,x_3\}$, $J_1 = \{t_1\}$, $I_2 = \{x_4,x_5,x_6,x_7,x_8\}$,
and $J_2 = \{t_2,t_3,t_4,t_5,t_6,t_7\}$, with squared ratios
$\frac{9}{9} = \frac{14}{14}$.
\end{ex}

Iteratively applying Lemma~\ref{l:residual} to the resulting graphs
$G(I_1,J_1)$ and $G(I_2,J_2)$ can produce various replacement
subsequences for $(A_i,B_i)$, depending upon the choice of min cuts
and the number of times the lemma is applied.  Picard and Queyranne
\cite{picard_queyranne} give a method to find all cuts for this flow
problem, thereby allowing us to characterize all ratio subsequences
associated with $(A_i,B_i)$.

\begin{defn}
Write $G_i^*(X)$ for the result of modifying the residual graph
$G_i^r$ by contracting all edges contained in directed cycles.
\end{defn}

The directed graph $G_i^*(X)$ is acyclic, is independent of the actual
(max) flow~$f$, and has nodes corresponding to a partition of the
nodes of $A_i\cup B_i\cup\{\bar s,\bar t\}$.  Furthermore, the nodes
in any partition obtained by iteratively applying
Lemma~\ref{l:residual} must consist of unions of the sets
corresponding to the nodes of $G_i^*$.

\begin{defn}
An \emph{upper ideal}\/ for~$G_i^*(X)$ is any set~$I$ of nodes
of~$G_i^*(X)$ such that $v \in I$ whenever $u \in I$ and $(u,v)$ is an
edge of~$G_i^*(X)$.
\end{defn}

A partition $(H,\oH)$ is therefore a cut if and only if $H$ is an
upper ideal.  Let ${\mathcal I}_i$ denote the set of upper ideals
of~$G_i^*(X)$, excluding the trivial ideal~$\{\bar s\}$.  The next
corollary follows from this discussion.

\begin{cor}\label{c:path_ordering}
The maximum size of any ratio subsequence that can replace $(A_i,
B_i)$ in~(\ref{eq:strict_inequalities}) is equal to the number of
vertices in $G_i^*(X)\setminus\{\bar s,\bar t\}$.  Moreover, the ratio
subsequences
$$%
\displaystyle
\frac{\norm{A_{i,1}'}}{\norm{B_{i,1}'}}
  = \frac{\norm{A_{i,2}'}}{\norm{B_{i,2}'}}
  = \cdots
  = \frac{\norm{A_{i,\ell}'}}{\norm{B_{i,\ell}'}}
$$
are in bijection with nested sequences of sets in ${\mathcal I}_i$.\qed
\end{cor}

This simplifies further.  A \emph{topological ordering}\/ of
$G_i^*(X)$ is any numbering of the vertices so that for every edge
$(u,v)$ of~$G_i^*(X)$, vertex~$v$ is numbered lower than~$u$.  Every
acyclic graph has at least one topological ordering.

\begin{cor}\label{c:topological_orderings}
The maximum-cardinality ratio subsequences of
Corollary~\ref{c:path_ordering} are in bijection with the topological
orderings of $G_i^*(X)$.  In fact, any ratio subsequence for a
particular pair $(A_i,B_i)$ corresponds to a partition of the vertices
of $G_i^*(X)$ according to one of these topological orderings.
\end{cor}

\begin{ex}[continuation of Example~\ref{e:picard_queyranne}]\label{e:cont}
Applying Corollary~\ref{c:topological_orderings} to the example in
Figure~\ref{f:picard_queyranne}, the only two acyclic orderings
of~$G^*$ are $(U,V,X,W)$ and $(U,V,W,X)$, which results in the two
maximal subsequences
$$%
(\cA,\cB)
  =
  \left\{\begin{array}{@{}l}
    \big(\{x_1,x_2,x_3\},\{t_1\}\big),
    \big(\{x_4,x_5,x_6\},\{t_2,t_3,t_4,t_5\}\big),
    \big(\{x_7\},\{t_6\}\big),
    \big(\{x_8\},\{t_7\}\big)
  \\[.5ex]
    \big(\{x_1,x_2,x_3\},\{t_1\}\big),
    \big(\{x_4,x_5,x_6\},\{t_2,t_3,t_4,t_5\}\big),
    \big(\{x_8\},\{t_7\}\big),
    \big(\{x_7\},\{t_6\}\big)
% \big((\{x_1,x_2,x_3\},\{t_1\}),
%     (\{x_4,x_5,x_6\},\{t_2,t_3,t_4,t_5\}),
%     (\{x_7\},\{t_6\}),
%     (\{x_8\},\{t_7\})\big)
% \\
% \big(
%     (\{x_1,x_2,x_3\},\{t_1\}),
%     (\{x_4,x_5,x_6\},\{t_2,t_3,t_4,t_5\}),
%     (\{x_8\},\{t_7\}),
%     (\{x_7\},\{t_6\})
% \big)
  \end{array}\right.
$$
respectively, both of which have squared ratios of $\frac 99 =
\frac{12}{12} = \frac 11 = \frac 11 = 1$.  The set of possible
replacement subsequences for $(\cA,\cB)$ corresponds to the twelve
distinct contiguous partitions that can be formed from one of the
above two sequences.
\end{ex}

\subsubsection{Valid support sequences}\label{ss:valid}%%%%%%%%%%%%%%%

We next set up the combinatorial structure to give a canonical
description of the vistal cell $\cV^2(T, \Or; \cA, \cB; \cS)$.

\begin{defn}\label{d:valid}
Let $(A_i,B_i)$ be a support pair for the minimal support $(\cA,
\cB)$.  A \emph{valid support sequence}\/ for $(A_i,B_i)$ is comprised
of a set of pairs
$(A'_{i,1},B'_{i,1}),\ldots,(A'_{i,\ell},B'_{i,\ell})$ with the
following properties.
\begin{enumerate}[\quad\rm (F1)]\setlength\itemsep{-.5ex}
\item%
The sets $A'_{i,j}$ and $B'_{i,j}$ are nonempty and partition $A_i$
and~$B_i$, respectively.
\item%
The incompatibility graph $G(A'_{i,j},B'_{i,j})$ is connected for each
$j=1,\ldots,\ell$.
\item%
Contracting the sets $A'_{i,j}\cup B'_{i,j}$ in $G(A_i,B_i)$ results
in an acyclic graph.
\end{enumerate}
\end{defn}

\begin{ex}[continuation of Examples~\ref{e:picard_queyranne} and~\ref{e:cont}]
Any support derived from the maximal supports in Example~\ref{e:cont}
is a valid support sequence, except for the two supports
\begin{align*}
& \big(\{x_1,x_2,x_3\},\{t_1\}\big),
  \big(\{x_4,x_5,x_6\},\{t_2,t_3,t_4,t_5\}\big),
  \big(\{x_7,x_8\},\{t_6,t_7\}\big)
\\\text{and }
& \big(\{x_1,x_2,x_3,x_4,x_5,x_6\},\{t_1,t_2,t_3,t_4,t_5\}\big),
  \big(\{x_7,x_8\},\{t_6,t_7\}\big),
\end{align*}
whose final pairs do not correspond to connected subgraphs of the
compatibility graph.
\end{ex}

\begin{lem}\label{l:valid_support_vector}
Let $X\in\cT_n$ have associated $(X,T)$-geodesic with minimal support
$(\cA, \cB)$ satisfying~(\ref{eq:strict_inequalities}), and for some
index~$i$ let $(A'_{i,1},B'_{i,1}),\ldots, (A'_{i,\ell},B'_{i,\ell})$
be a valid support sequence for $(A_i,B_i)$.  There is an element
$X'\in \cT_n$ 
%with $\Or(X')=\Or(X)$ 
in the same orthant as $X$ for which the geodesic between
$X'$ and $T$ has support\vspace{-.5ex}
\begin{align*}
  \cA'
& = A_1,\ldots,A_{i-1},A'_{i,1},\ldots,A'_{i,\ell},A_{i+1},\ldots,A_k
\\\cB'
& = B_1,\ldots,B_{i-1},B'_{i,1},\ldots,B'_{i,\ell},B_{i+1},\ldots,B_k
\\[-5ex]
\end{align*}
with
%begin{equation}\label{eq:expanded_sequence}
$$%
  \frac{\norm{A_1}}{\norm{B_1}}
< \cdots
< \frac{\norm{A_{i-1}}}{\norm{B_{i-1}}}
< \frac{\norm{A_{i,1}'}}{\norm{B_{i,1}'}}
= \cdots
= \frac{\norm{A_{i,\ell}'}}{\norm{B_{i,\ell}'}}
< \frac{\norm{A_{i+1}}}{\norm{B_{i+1}}}
< \cdots
< \frac{\norm{A_k}}{\norm{B_k}}.
$$
%end{equation}
Further, for any pair $(A'_{i,j}, B'_{i,j})$ and any partition $I_1
\cup I_2$ of~$A_{i,j}'$ and~$J_1\cup J_2$ of~$B'_{i,j}$ in which
$I_2\cup J_1$ is compatible,
%begin{equation}\label{eq:expanded_sequence2}
$$%
  \frac{\norm{I_1}}{ \norm{J_1}} > \frac{ \norm{ I_2}}{ \norm{J_2} }.
$$
%end{equation}
\end{lem}
\begin{proof}
%First, if $(A_i, B_i)$ is a support pair corresponding to tree edges
%common to $T$ and $X$, then by (F2) the only valid support
%subsequences for $(A_i, B_i)$ will have $|A'_{i,j}|=|B'_{i,j}| = 1$,
%and so $X' = X$ satisfies (\ref{eq:expanded_sequence}) and
%(\ref{eq:expanded_sequence2}) vacuously.
%Now let $(A_i,B_i)$ be a nontrivial support pair.
For support pair $(A_i, B_i)$, let $\tilde\xi$ and $\tilde\tau$ be the
weights on the vertices of $G(A_i,B_i)$.  Define $X'$ by replacing the
(squared) weights on $X$ for each $a \in A'_{i,j}$ by
$$%
  \tilde\xi'_a = \sum_{b\in E_j(a)}\frac{\tilde\tau_b}{\deg_j(b)},
$$
where $E_j(a)$ is the set of vertices $b\in B'_{i,j}$ such that
$(a,b)$ is in the incompatibility graph, and $\deg_j(b)$ is the number
of edges of the incompatibility graph from $A'_{i,j}$ to~$b$.  These
values are all well-defined and positive by (F1) and (F2).  Place the
following flow $f$ on the associated flow graph: for edge $(a,b)$
where $a \in A'_{i,j}$ and $b \in B'_{i,j}$ for any $1 \leq j \leq l$,
let the flow on that edge be $\tilde\tau_b/\deg_j(b)$; for all other
edges, let the flow be~$0$.  Then the flow into node~$b$ is
exactly~$\tilde\tau_b$ and the flow out of~$a$ is
exactly~$\tilde\xi_a$.  Corollary~\ref{c:path_ordering} and
property~(F3) ensure that $f$ is a max flow with respect to the flow
graph, with flow value $\sum_{B_i}\tilde\tau_b = \sum_{A_i}\tilde\xi_a
= 1$, and since flow is conserved between each $A'_{i,j}$
and~$B'_{i,j}$, the original (un-normalized) weights satisfy
$$%
  \frac{\norm{A_{i,1}'}}{\norm{B_{i,1}'}}
= \cdots
= \frac{\norm{A_{i,\ell}'}}{\norm{B_{i,\ell}'}}
= \frac{\norm{A_i}}{\norm{B_i}}.
$$
Finally, for a pair $(A'_{i,j},B'_{i,j})$, let $I_1\cap I_2$ and
$J_1\cap J_2$ be partitions of $A'_{i,j}$ and $B'_{i,j}$ respectively,
in which $I_2\cup J_1$ is compatible.  This means that there are no
edges of $G(A'_{i,j},B'_{i,j})$ from $I_2$ to~$J_1$, and since
$G(A'_{i,j},B'_{i,j})$ is connected there must be at least one edge
going from $I_1$ to~$J_2$.  Since flow is positive on all edges of
$G(A'_{i,j},B'_{i,j})$, there is a net flow from~$I_1$ away
from~$J_1$, and from the definition of $\xi'$ it follows that
$\frac{\norm{I_1}}{\norm{J_1}} > \frac{\norm{I_2}}{\norm{J_2}}$.
\end{proof}

\subsubsection{Canonical description of vistal cells}\label{ss:canonical}

Finally, we extend Propositions ~\ref{p:cone} and \ref{p:polyhedral}
to describe all vistal cells associated with $(X,T)$-geodesics from
points~$X$ in an orthant~$\Or$.  Since a valid support sequence is
determined by the combinatorics of the splits and not by their edge
lengths, we can define the following.

\begin{defn}
A \emph{valid support sequence for $(\Or,T)$} is a support $(\cA,\cB)$
for which each maximal equality subsequence
\begin{equation}\label{eq:equality_subsequence}
\frac{\norm{A_i}}{\norm{B_i}}
= \frac{\norm{A_{i+1}}}{\norm{B_{i+1}}}
= \cdots
= \frac{\norm{A_j}}{\norm{B_j}}
\end{equation}
satisfies properties (F1)\textendash{}(F3) with respect to the pair
$(\bigcup_{\ell=i}^j A_\ell,\bigcup_{\ell=i}^j B_\ell)$.  Write
$G(\Or,T)$ for the corresponding incompatibility graph $G(\cA,\cB)$.
\end{defn}

\begin{thm}\label{t:vistal_face_combinatorics}
Fix a tree $T \in \cT_n$.\vspace{-1ex}
\begin{enumerate}\setlength\itemsep{-.5ex}
\item%
Vistal cells associated with geodesics to~$T$ are exactly those of the
form $\cV^2(T, \Or; \cA, \cB; \cS)$, where $(\cA,\cB)$ is a valid
support sequence for $(\Or,T)$ and $\cS$ is a signature on
$(\cA,\cB)$.
\item%
The dimension of the vistal cell $\cV^2(T, \Or; \cA, \cB; \cS)$ is
$\dim(\Or) - m(\cS)$, where $m(\cS)$ is the number of~``='' components
in~$\cS$.
\item%
The representation by a valid support sequence and signature is unique
up to reordering the support sets within each equality subsequence
of~$\cS$.
\end{enumerate}
\end{thm}
\begin{proof}
Claim~1.  Let $\cV^2(T, \Or; \cA, \cB; \cS)$ be a vistal cell
containing an interior point~$\xi$.  The definition of support and the
fact that $\xi$ is positive implies that (F1) and~(F3) hold for
$G(\Or,T)$.  Now suppose that (F2) fails to hold; that is, some
$G(A_i,B_i)$ has a partition into two disjoint subgraphs on vertex
sets $I_1\cup J_1$ and $I_2\cup J_2$, respectively.  Let $f$ be the
max flow in $G(A_i,B_i)$.  Since (P3) is satisfied, $f$ saturates all
arcs adjacent to the source and sink.  But since flow in each of the
disjoint subgraphs $G(I_1,J_1)$ and $G(I_2,J_2)$ is self-contained,
$\frac{\norm{I_1}}{\norm{J_1}} = \frac{\norm{I_2}}{\norm{J_2}} =
\frac{\norm{A_i}}{\norm{B_i}}$.  This means that the corresponding
tree~$X$ satisfies one of its (P3) inequalities at equality, so $\xi$
cannot be in the interior of $\cV^2(T, \Or; \cA, \cB; \cS)$, a
contradiction.  Thus (F2) is also satisfied, so $(\cA,\cB)$ is a valid
support sequence with respect to~$(\Or,T)$.

Conversely, let $(\cA,\cB)$ be a valid support sequence with respect
to $(\Or,T)$.  Consider a ratio
subsequence~(\ref{eq:equality_subsequence}) with all terms equal.
Since $(\cA,\cB)$ is a valid support sequence,
Lemma~\ref{l:valid_support_vector} constructs positive weights
$X^\ell$ on the edges indexed by~$A_\ell$, for $\ell = i,\ldots,j$, so
that~(\ref{eq:equality_subsequence}) holds and all~(P3) inequalities
are strict inside each support pair.  Now for each maximal-length
equal-ratio subsequence, scale the vectors of each term by the same
positive multiplier~$\lambda_{ij}$ so that the sequence of
multipliers~$\lambda_{ij}$ is increasing with the indices.  The scaled
$x^\ell$ vectors concatenate into a vector $X$ in the interior of
$\Or$ having the correct signature indicated by~$\cS$, and for which
the (P2) inequalities hold strictly between the equal-ratio
subsequences.  The squared point~$\xi$ corresponding to~$X$ therefore
lies interior to $\cV^2(T, \Or; \cA, \cB; \cS)$, and the desired
result follows.

Claim~2.  The vector~$\xi$ constructed in the proof of Claim~1 is
positive in~$\Or$, satisfies all (P3) inequalities strictly, and
satisfies all (P2) inequalities strictly for which the corresponding
component of $\cS$ is~``$\leq$''.  Therefore the dimension of
$\cV^2(T, \Or; \cA, \cB; \cS)$ is determined entirely by the set of
equalities defined by the component of~$\cS$ that are~``=''.  Since
these are linearly independent, the dimension is as stated.

Claim~3.  Let $F = \cV^2(T, \Or; \cA, \cB; \cS)$ and $F' = \cV^2(T,
\Or'; \cA', \cB'\cS')$ be two representations of vistal cells, defined
by valid supports $(\cA, \cB)$ and $(\cA', \cB')$ respectively.  Any
permutation of support pairs within an equality
subsequence~(\ref{eq:equality_subsequence}) results in the same set of
equalities, so if the representations differ only by such a
permutation, then $F = F'$.  Conversely, suppose $F = F'$.  Since all
cell constraint inequalities other than those specified by $\cS$ are
satisfied strictly, the set of equalities dictated by~$\cS$ define the
affine hulls of~$F$ and~$F'$.  This means that the two associated
equality systems are row-equivalent.  Now suppose that the supports
$(\cA, \cB)$ and $(\cA', \cB')$ do not comprise the same sets; that
is, by symmetry the two sets $A_i$ and~$A_j$ both have nonempty
intersection with the same set~$A'_k$.  Since the variables of $A'_k$
do not appear in any other $A'_\ell$ for $\ell \neq k$, no row
transformation of the equality system for $F'$ could possibly separate
the variables in $A_i\cap A'_k$ from those in $A_j\cap A'_k$.  Thus
the two equality systems are not the same, a contradiction.
\end{proof}
\begin{cor}\label{c:disjoint}
Distinct vistal cells have disjoint relative interiors.
\end{cor}
\begin{proof}
Let $\xi$ be an element in the relative interior of two faces
in~$\cT^2_n$, given by valid representations.
% $\cV^2(T, \Or; \cA, \cB; \cS)$ and $\cV^2(T', \Or'; \cA', \cB';
% \cS')$.
Then $\xi$ satisfies (F2) and (F3) with respect to both faces, and by
Theorem~\ref{t:vistal_face_combinatorics} the only way this could
happen is if the faces coincide.
\end{proof}

%%%%%%%%%%%%%%%%%%%%%%%%%%%%%%%%%%%%%%%%%%%%%%%%%%%%%%%%%%%%%%%%%%%%%%
\subsection{Vistal subdivisions}\label{sub:vistalSub}%%%%%%%%%%%%%%%%%

Theorem~\ref{t:vistal_face_combinatorics} allows us a purely
combinatorial way of describing vistal cells.  This gives us the
machinery to prove the principal result of the section, namely that
the vistal cells are the faces of a polyhedral subdivision of tree
space under the squaring map.  To make this precise, we start with
some definitions concerning polyhedra; see \cite[Lecture~5]{ziegler}
for further background.

\begin{defn}\label{d:polyhedralComplex}
A \emph{polyhedral complex} $\Sigma$ is a finite
% nonempty
collection of polyhedra
% called \emph{cells} of~$\Sigma$,
such that
\begin{enumerate}[\quad\rm(C1)]
%item%
%Each cell is a polyhedron (actually a polytope).
%item%
%The union of the cells in $\Sigma$ is all of $L_n$.
%item%
%The relative interiors of different cells are disjoint.
\item%
every polyhedral face of every polyhedron in~$\Sigma$ is a polyhedron
in~$\Sigma$;
\item%
the intersection of any pair of polyhedra in~$\Sigma$ is a face of each.
\end{enumerate}
The \emph{dimension}\/ of~$\Sigma$ is the largest dimension of a
polyhedron in~$\Sigma$.  The \emph{facets}\/ of~$\Sigma$ are the
maximal cells.  The \emph{underlying set} of~$\Sigma$ is the union
$\bigcup_{V \in \Sigma} V$ of the polyhedra in~$\Sigma$.
\end{defn}

\begin{ex}\label{e:treeSigma}
Tree space $\cT_n$ has a \emph{natural polyhedral structure} as the
underlying space of a polyhedral complex whose polyhedra are its
orthants.  This polyhedral structure is unchanged by the squaring map,
and thus also found in $T^2_n$.
\end{ex}

The relation between vistal cells and orthants is one of refinement,
in the following sense.

\begin{defn}\label{d:subdivision}
Let $\Sigma$ and~$\Sigma'$ be polyhedral complexes.  Then $\Sigma'$ is
a \emph{subdivision} of~$\Sigma$ (it is also said that $\Sigma'$
\emph{refines}~$\Sigma$) if each polyhedron in~$\Sigma'$ is contained
in a single polyhedron in~$\Sigma$.
\end{defn}

\begin{thm}\label{t:cell_complex}
For tree space~$\cT_n$ and fixed source tree $T$, the vistal cells of
$\cT^2_n$ with respect to $T$ refine the natural
% $(2n-1)$-dimensional
polyhedral structure of~$\cT^2_n$ to form a \emph{vistal
polyhedral subdivision} of~$\cT^2_n$.
\end{thm}
\begin{proof}
% old (C1) follows from Propositions~\ref{p:cone}.
The vistal cells are polyhedra whose union is~$\cT^2_n$ by
Propositions~\ref{p:cone} and~\ref{p:polyhedral}.

% old (C2) follows from the definition of $\Sigma$.

% \comment{EM: old (C3) follows formally from old (C5).  That said,
% the disjointness seems to be needed for the proof of new (C2), so
% I've simply moved this paragraph to the relevant place (the end of
% the proof).}

% \comment{EM: (C1) is old (C4)}
For (C1), we show that changing any of the inequalities defining a
vistal cell to equality results in a set that can be described as a
vistal cell.  Let $V = \cV^2(T, \Or; \cA, \cB; \cS)$ be a vistal cell,
so that by Lemma~\ref{l:valid_support_vector}, $(\cA, \cB)$ is a valid
support sequence, and let~$F$ be a proper face of~$V$ obtained by
setting one of its boundary inequalities to equality.  There are three
types of inequalities that define~$F$: (P2) constraints, nonnegativity
constraints, and (P3) constraints.

For the (P2) constraints, consider the inequality
$\frac{\norm{A_i}}{\norm{B_i}} <
\frac{\norm{A_{i+1}}}{\norm{B_{i+1}}}$, where the corresponding
component of the signature~$\cS$ is~``$\leq$''.  Let $\cS'$ be
obtained from $\cS$ by setting this inequality to~``=''.  Since
neither $\Or$ nor $(\cA, \cB)$ has changed, this constitutes a valid
support sequence, and $F = \cV^2(T, \Or; \cA, \cB; \cS')$.

For the nonnegativity constraints, consider the inequality $x_e > 0$,
where $e$ is a split indexing a coordinate of~$\Or$.  Let $A_i$ be the
set containing $e$.  Now remove $e$ from $G(\Or,T)$.  This splits
$G(A_i,B_i)$ into components corresponding to partitions
$(A'_1,B'_1),\ldots,(A'_\ell,B'_\ell)$ of $(A_i,B_i)$.  Because these
partitions correspond to separate components in $G(A_i, B_i)$, they
can appear in any order in a valid support sequence for $F$.
% without violating~(P2).
Thus every point in $F$ must satisfy every (P2) inequality between the
pairs $(A'_i,B'_i)$ at equality, since otherwise the $(A'_j,B'_j)$
sets could be interchanged so that some (P3) condition is violated.
First consider the case where all of the $A'_j$ are nonempty.  Define
the support $(\cA',\cB')$ by inserting $(A'_1,B'_1), \ldots,
(A'_\ell,B'_\ell)$ in place of $(A_i,B_i)$ in $(\cA,\cB)$:
\begin{align*}
  (\cA',\cB')
  = (A_1,B_1),\ldots,(A_{i-1},B_{i-1}),(A'_1,B'_1),& \ldots,
  (A'_\ell,B'_\ell),\\
  &(A_{i+1},B_{i+1}),\ldots,(A_k,B_k)
\end{align*}
and extend the signature~$\cS$ to~$\cS'$ by adding ``='' signs between
each of the sets in the primed subsequence.  Then $(\cA',\cB')$ is
valid, and $F = \cV^2(T, \Or\setminus\{e\}; \cA', \cB'; \cS')$.

Now suppose that one of the support pairs $(A'_j,B'_j)$ has
$A'_j=\nothing$.  The associated ratio must be~$0$, which implies in
turn that every ratio corresponding to the pairs
$(A'_1,B'_1),\ldots,(A'_\ell,B'_\ell)$ is~$0$.  Furthermore, the
ratios are also~$0$ for any earlier support pairs.
%Let $p$ be the smallest index of a support pair not containing common edges.
So $x_f = 0$ for every $f \in \oH_i = %A_p
A_1 \cup \cdots \cup A_i$.  In this case set\vspace{-.25ex}
\begin{align*}
\ol\Or{}'
&= \Or\setminus \oH_i
\\
(\ol\cA{}',\ol\cB{}')
%&= (A_1, B_1), \ldots, (A_{p-1},B_{p-1}),(\nothing, B_p \cup \cdots
 %  \cup B_i),(A_{i+1},B_{i+1}),\ldots,(A_k,B_k)
&=(A_{i+1},B_{i+1}),\ldots,(A_k,B_k)
\\
\ol\cS{}'
%&= \text{the first $p-1$ terms of~$\cS$, then ``$<$'' followed by
&=\text{$\cS$ restricted to the last $k - i$ pairs of the sequence.}
\end{align*}
By Remark~\ref{r:common_edges} we have been ignoring the non-positive
ratios; however, they still exist if there are common edges
between~$X$ and~$T$.  In this case, the edges $B_1 \cup \cdots \cup
B_i$ become common edges, and are added to the $0$-valued ratio if it
already exists, or form it anew, if it does not.  Again
$(\ol\cA{}',\ol\cB{}')$ is valid, and $F=\cV^2(T, \ol\Or{}';
\ol\cA{}', \ol\cB{}'; \ol\cS{}')$.

Next consider the (P3) constraints.  For some support pair $(A_i,B_i)$
let $I_1\cup I_2$ and $J_1\cup J_2$ be partitions of $A_i$ and $B_i$
with $I_2\cup J_1$ compatible, and consider the constraint
$$%
  \frac{\norm{I_1}}{\norm{J_1}} > \frac{\norm{I_2}}{\norm{J_2}}.
$$
Let $(A'_1,B'_1),\ldots,(A'_k,B'_k)$ and
$(A'_{k+1},B'_{k+1}),\ldots,(A'_\ell,B'_\ell)$ be pairs corresponding
to the components of $G(I_1,J_1)$ and $G(I_2,J_2)$, respectively.

First consider the case where all of the $A'_j$ and~$B'_j$ are
nonempty.  The same nonempty sets argument as above applies, and we
obtain the the face $F = \cV^2(T, \Or; \cA', \cB'; \cS')$ with $(\cA',
\cB')$ and~$\cS'$ defined as in the nonempty-set case above.

Next suppose that one of the sets $(A'_j,B'_j)$ has $A'_j=\nothing$.
As in the empty-set case above, this forces $x_f$ to be 0 for every
% $f\in S_i=A_j\cup\cdots\cup A_i$,
$f \in S_i = A_1 \cup \cdots \cup A_i$,
% for the $j >0 $ such that $(A_{j-1},B_{j-1})$ contains common edges,
% but $(A_j, B_j)$ does not,
and so $F = \cV^2(T, \ol\Or{}'; \ol\cA{}', \ol\cB{}'; \ol\cS{}')$
with~$\ol\Or{}'$, $(\ol\cA{}', \ol\cB{}')$ and~$\ol\cS{}'$ defined as
in the empty-set case above.

Now suppose that one of the sets $(A'_j,B'_j)$ has $B'_j = \nothing$.
This forces the ratios for every pair in
$(A'_1,B'_1),\ldots,(A'_\ell,B'_\ell)$ to be $\infty$, which in turn
means that $x_f=0$ for every $f\in \tilde S'_i=B_{i+1}\cup\cdots\cup
B_k$.  Thus if we define
\begin{align*}
\wt\Or'
&= \Or\setminus \tilde S'_i
\\
(\tilde\cA',\tilde\cB')
&= (A_1,B_1),\ldots,(A_{i-1},B_{i-1})
%,(A_i\cup\cdots\cup A_k,\nothing)
\\
\tilde\cS'
&= \cS\text{ restricted to the first }i-1\text{ pairs of the sequence,}
%   followed by ``$<$'',}
\end{align*}
then again $(\tilde\cA',\tilde\cB')$ is a valid sequence, and so we
obtain the face
 $F=\cV^2(T, \tilde\Or';\allowbreak \tilde\cA', \tilde\cB';\cS')$.  As before, the edges $A_i \cup \cdots \cup A_k$ become common
edges, and hence be added to the $\infty$-valued ratio if it exists
and otherwise form that ratio.

Finally, suppose that there are pairs $(A'_{j'},B'_{j'})$ and
$(A''_{j''},B''_{j''})$ with $A'_{j'}=B''_{j''}=\nothing$.  This
forces all of the $x_f$ where $f$ is not a common edge to be~$0$, and
we just get the face corresponding to the common edges.

% \comment{EM: (C2) is old (C5)}
% For (C2), suppose that $V$ and~$V'$ are vistal cells.  Since $V$ and
% $V'$ are polyhedra, their intersection $V \cap V'$ is also a
% polyhedron, say of dimension~$k$.  \comment{EM: I don't understand
% the following deduction:} $V \cap V'$ is therefore contained in
% $k$-dimensional faces $F$ and~$F'$ of $V$ and~$V'$, respectively,
% which by (C4) are vistal cells \comment{EM: why can't $F$ and~$F'$
% have dimension strictly bigger than~$k$---for example, two planar
% disks in 4-space can intersect at a point that lies in the relative
% interior of both}.

For (C2), suppose that $V$ and~$V'$ are vistal cells, so that $V \cap
V'$ is a convex polyhedron.  Let $F \subseteq V$ and $F' \subseteq V'$
be minimal faces of $V$ and~$V'$, respectively, containing $V \cap
V'$.  Then by (C1), $F$ and~$F'$ are vistal cells, and since $F$ and
$F'$ are minimal, then there must be a $p \in V \cap V'$ in the
relative interior of~$F$ and a $p' \in V \cap V'$ in the relative
interior of~$F'$.  It follows that the midpoint of the line segment
joining $p$ to~$p'$ must lie in the relative interiors of both~$F$
and~$F'$, and Corollary~\ref{c:disjoint} then implies that $F = F'$.
Thus $F = F' \subseteq V \cap V'$, whence $V \cap V'= F = F'$ and (C2)
follows.
\end{proof}

\subsection{Examples of vistal complexes}\label{sub:vistalEx}%%%%%%%%%

\begin{ex}\label{e:cell}
To demonstrate Theorem~\ref{t:cell_complex}, consider the
incompatibility graph from Figure~\ref{f:picard_queyranne} and treat
it as the incompatibility graph for two trees $T$ and~$X$.  Take
values on~$T$ as given in the figure, and consider the vistal cell $V
= \cV^2(T, \Or; \cA, \cB; \cS)$ defined by
\begin{align*}
\Or
&= \{x_1,x_2,x_3,x_4,x_5,x_6,x_7,x_8\}
\\
(\cA,\cB)
&= \big(\{x_1,x_2,x_3\}, \{t_1\}\big),
   \big(\{x_4,x_5,x_6,x_7,x_8\},\{t_2,t_3,t_4,t_5,t_6,t_7\}\big)
\\
\cS
&= (\leq).
\end{align*}
This is a valid sequence, and in particular, using
Lemma~\ref{l:valid_support_vector} we can assign weights as follows.
\[
  \begin{array}{|cccccccc|}
  \hline
  x_1&x_2&x_3&x_4&x_5&x_6&x_7&x_8\\
  \hline
  &&&&&&&\\[-.8em]
  2&2&2&3\frac{1}{3}&4\frac{1}{3}&5\frac{1}{2}&\frac{1}{3}&\frac{1}{2}\\[.2em]
  \hline
  \end{array}
\]
(The first three weights have additionally been scaled so that (P2) is
satisfied strictly.)  Here are examples of the three types of faces
of~$V$.
\begin{bullets}
\item%
Setting the single (P2) constraint to equality: this gives the face
corresponding to the numbers in Figure~\ref{f:picard_queyranne}.

\item%
Setting $x_j=0$: for $j \neq 5,6$ the face has the same structure as
the cell~$V$, except that $x_j$ is removed from the corresponding
sets.  For $j= 5,6$, removal of $x_j$ disconnects $(A_2,B_2)$ by
isolating $t_4$ or $\{t_3, t_5\}$,
respectively, and thus setting $x_5$ or $x_6$ to 0 collapses the face
to the single origin point.

\item%
Setting the (P3) constraint with $I_1=\{x_4,x_5,x_6\}$,
$J_1=\{t_2,t_3,t_4,t_5\}$, $I_2=\{x_7,x_8\}$, and $J_2=\{t_6,t_7\}$ to
equality: here
$$%
  \frac{\norm{I_1}^2}{\norm{J_1}^2}
  =
  \frac{79}{66}>\frac{5}{12}
  =
  \frac{\norm{I_2}^2}{\norm{J_2}^2}.
$$
Now $G(I_2,J_2)$ is not connected, and has nontrivial components on
vertex sets $\{x_7,t_6\}$ and $\{x_8,t_7\}$.  Thus the face obtained
by setting the above inequality to equality is $\cV^2(T, \Or; \cA',
\cB'; \cS')$, where
\begin{align*}
(\cA',\cB')
&= \big(\{x_1,x_2,x_3\}, \{t_1\}\big),
   \big(\{x_4,x_5,x_6\},\{t_2,t_3,t_4,t_5\}\big),
   \big(\{x_7\},\{t_6\}\big), \\
   &\qquad\qquad\qquad\qquad\qquad\qquad\qquad\qquad\qquad\qquad\qquad\qquad\big(\{x_8\},\{t_7\}\big)
\\
%(\cA',\cB')&=&\big((\{x_1,x_2,x_3\},
%\{t_1\}),(\{x_4,x_5,x_6\},\{t_2,t_3,t_4\}),(\{x_7\},\{t_6\}),(\{x_8\},\{t_7\})\big)\\
\cS'
&= (\leq,=,=).
\end{align*}\vspace{-4ex}
\end{bullets}
\end{ex}

\begin{ex}\label{e:subdivision}
Figure~\ref{f:example_complex} gives the restriction of a vistal
polyhedral subdivision to a maximal orthant in~$\cT_5$.  The trees are
depicted in Figure~\ref{f:original_trees}, with \mbox{$t_1 = t_2 = t_3
= 1$}.  Figures~\ref{f:vistal_sub_unsquared} and~\ref{f:vistal_sub}
\begin{figure}
\centering
\subfigure[Trees $X$ and $T$.]{
\includegraphics[trim=0cm 0cm 0cm 0.5cm,clip=true,scale=0.35]{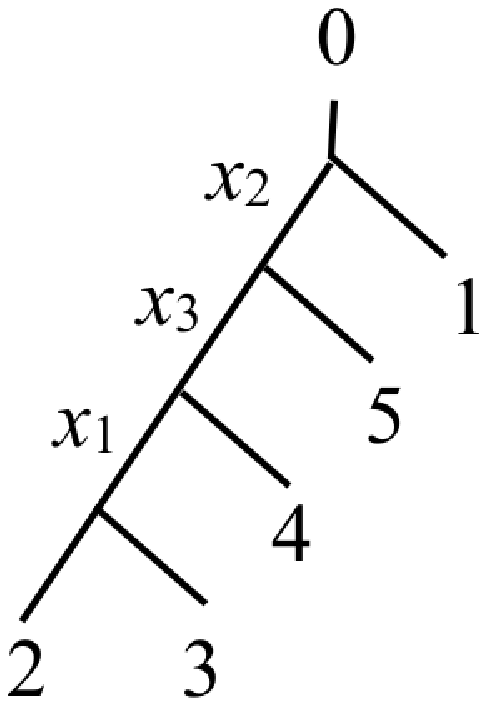}
\quad
\includegraphics[trim=0cm 0cm 0cm 0.5cm,clip=true,scale=0.35]{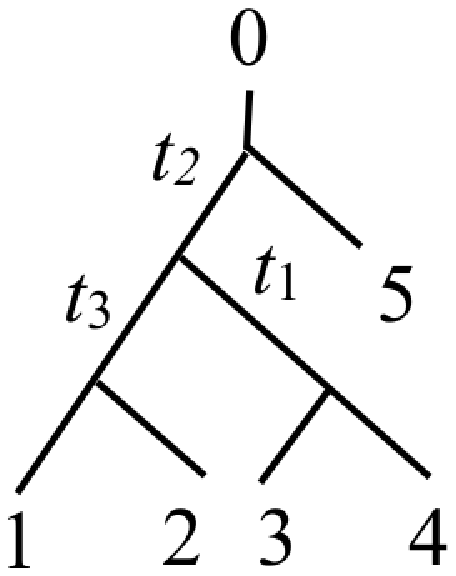}
\label{f:original_trees}
}
\subfigure[A cross-section
of the orthant corresponding to tree topology $X$ before the squaring map.]%
{\includegraphics[width=120mm]{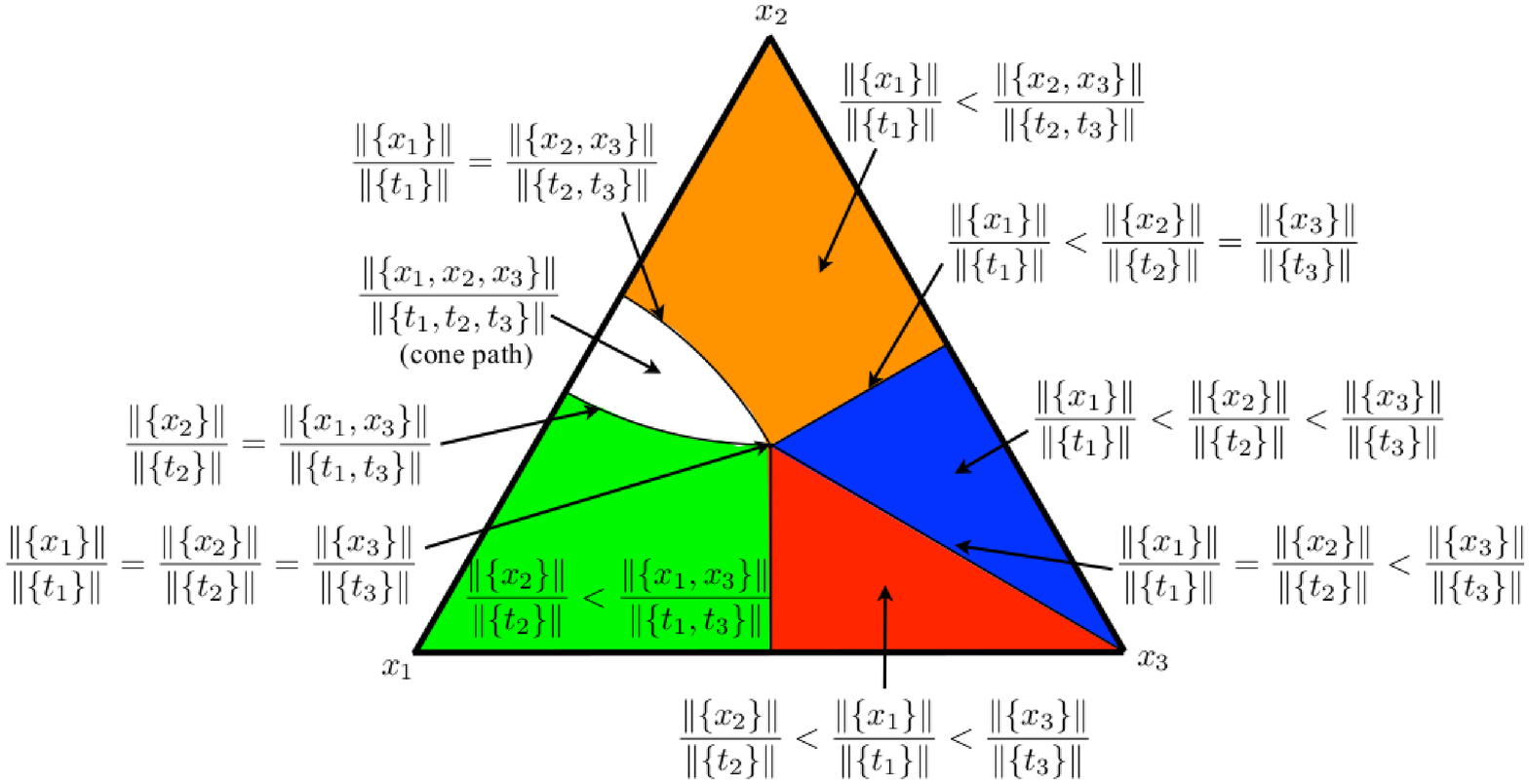}\label{f:vistal_sub_unsquared}}
\subfigure[A cross-section
of the orthant corresponding to tree topology $X$ under the squaring map.  Vistal cells are labelled as in Figure~\ref{f:vistal_sub_unsquared}]%
{\includegraphics[width=70mm]{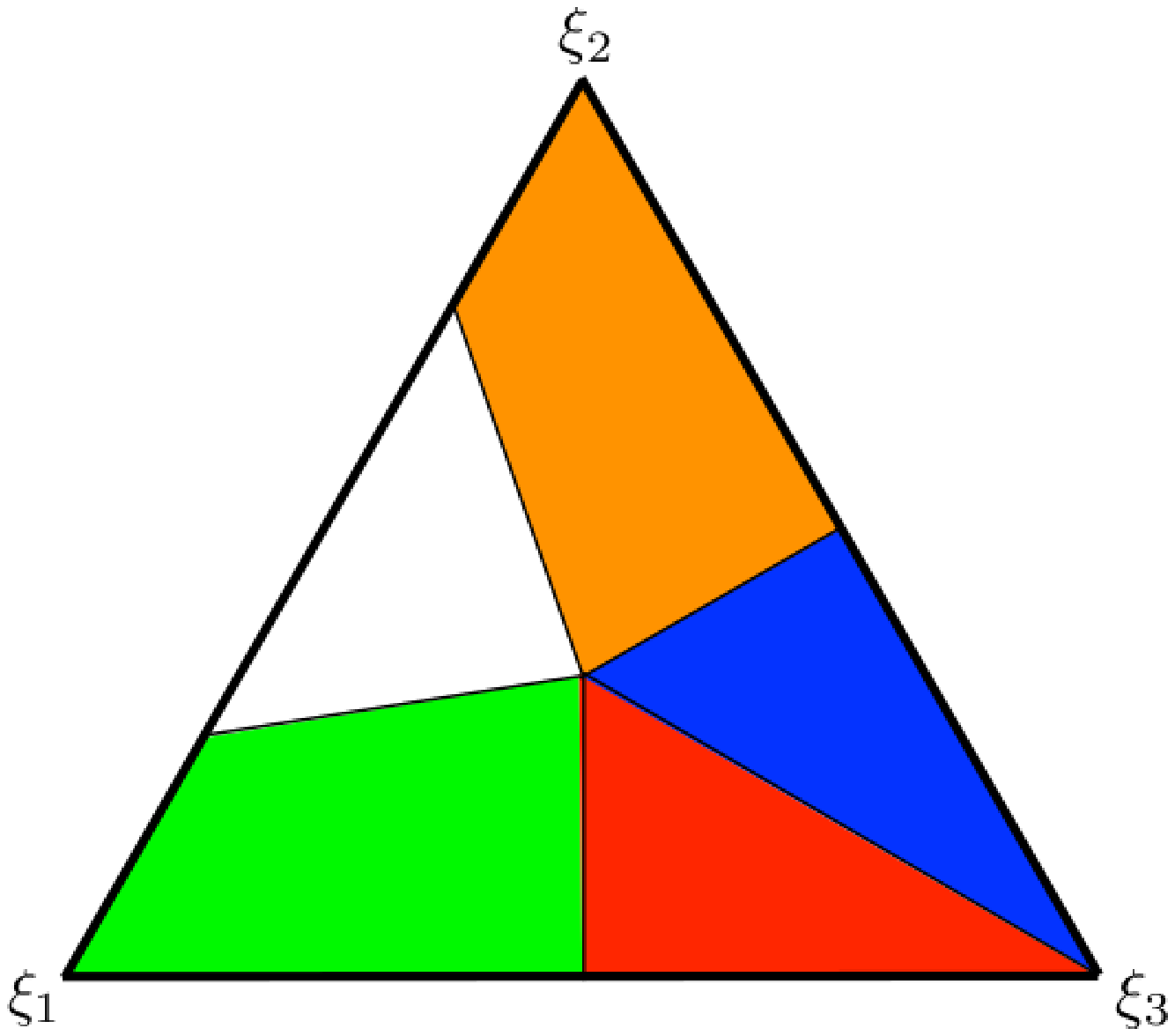}\label{f:vistal_sub}}
\caption{The vistal polyhedral subdivision between variable tree $X$ and fixed tree $T$
 in~$\cT_5$.\label{f:vistal_sub_all}}
\centering
\label{f:example_complex}
\end{figure}
depict the vistal cells in orthant $\Or(\{x_1,x_2,x_3\})$ before and under the 
squaring map, respectively,
as they intersect
with the sets \mbox{$x_1 + x_2 + x_3 = 1$} and
\mbox{$\xi_1 + \xi_2 + \xi_3 = 1$}, respectively.
The vistal cells are
labeled by the corresponding ratio sequences, using ``='' or
``$<$'' to indicate the behavior of points in the interior of the
cell.  We also label the six cells of lower dimension that are
the intersections of the vistal facets.
Some of the vistal cells before the squaring map are 
not polyhedral, because the boundary equations are those 
of circular cones.
\end{ex}

%%%%%%%%%%%%%%%%%%%%%%%%%%%%%%%%%%%%%%%%%%%%%%%%%%%%%%%%%%%%%%%%%%%%%%
\subsection{Multivistal complexes}\label{sub:multivistal}%%%%%%%%%%%%%

It is a straightforward matter to extend vistal cells to the case
where there is a collection ${\bm T} = \{T^1,\ldots,T^r\}$ of source
trees in $\cT_n$, and we are interested in the set of points $X\in
\cT_n$ for which the geodesic to each tree in ${\bm T}$ has a
specified combinatorial structure.
%This is a more relevant property when we are studying properties of
%the mean of the set ${\bm T}$.
\begin{defn}\label{d:multivistal_def}
A \emph{premultivistal cell} for a collection $\bm T$ of trees is a
set of the form
$$%
  \cV(\bm{T};\Or;\cA^{\bm T}, \cB^{\bm T}) =
  \bigcap_{\ell=1}^r \cV(T^\ell,\Or;\cA^\ell, \cB^\ell),
$$
where $\cV(T^\ell,\Or; \cA^\ell, \cB^\ell)$ are previstal cells, $\Or
\subseteq \cT_n$ is an orthant, and
$$%
  (\cA^{\bm T}, \cB^{\bm T})
  =
  \big\{(\cA^1,\cB^1)\ldots,(\cA^r,\cB^r)\big\}
$$
is a collection of support pairs for $(T^i,X)$-geodesics.  A
\emph{multivistal cell} is the image in~$\cT_n^2$ of a premultivistal
cell.
\end{defn}

\begin{cor}\label{c:multivistal}
% For a fixed set ${\bm T}$ of source trees, the collection of links
% of multivistal cells form a $(2n-1)$-dimensional polyhedral cell
% complex on $L_n$.
The multivistal cells of tree space~$\cT_n$ for any fixed set source
trees refine the natural polyhedral structure of~$\cT_n$ to form a
\emph{multivistal polyhedral subdivision} of~$\cT_n$.
\end{cor}
\begin{proof}
The common refinement of any finite collection of polyhedral
subdivisions of a given polyhedral complex is a polyhedral subdivision
of the same polyhedral complex, and so the result follows from
Theorem~\ref{t:cell_complex}.
\end{proof}

% We define the \emph{multivistal complex} as the fan of cones
% generated by the complex given in Corollary~\ref{c:multivistal}.

\begin{rem}\label{r:cw}
The previstal cells for any fixed source tree form a subdivision
of~$\cT_n$, called a \emph{premultivistal complex}, that is the image
of the corresponding multivistal polyhedral subdivision of~$\cT^2_n$
under the inverse $\xi \to \sqrt\xi$ of the squaring map, which is a
homeomorphism.  However, the
% (links of the)
cells in this subdivision are not polyhedral.  One might hope that the
premultivistal complex is a \emph{CW~complex}, in the standard
topological sense (see \cite{munkres}, for example),
% A \emph{CW~complex} is a collection of \emph{cells}, with each cell
% homeomorphic to an $n$-ball for some $n$, and which satisfy
% conditions (C2), (C3), and (C5).  (See \cite{munkres} for
% definitions and additional background.)
but it is not, for the same reason that multivistal polyhedral
subdivisions are not CW~complexes: the closed cells are not images of
closed balls under continuous maps (a cone of positive dimension fails
to be compact).  The situation can be
% rectified
remedied by considering the \emph{link} $L_n$ of the origin
in~$\cT^2_n$, namely the set of trees whose edge lengths sum to~$1$.
Intersecting $L_n$ with any multivistal polyhedral subdivision yields
a polyhedral CW~complex whose preimage under the squaring map is a
(non-polyhedral) CW~complex.  Thus a premultivistal complex is
essentially a (noncompact, unbounded) cone over a CW~complex.
% the failure of a premultivistal complex to be a CW~complex is a
% homeomorphic image of the same failure of the multivistal polyhedral
% subdivision to be a CW~complex.
\end{rem}

% \begin{cor}\label{c:premultivistal}
% For a fixed set ${\bm T}$ of source trees, the collection of
% premultivistal cells form the fan of a $(2n-1)$-dimensional CW~complex
% on (the link of the origin of) $\cT_n$.
% \end{cor}
% \begin{proof}
% The inverse of the squaring map $\xi\to \sqrt{\xi}$ is a homeomorphism
% on each orthant, and hence all previstal cells are homeomorphic to
% closed balls \comment{no, unbounded polyhedra aren't closed balls}.
% \end{proof}

\begin{rem}\label{r:vistalExp}
In general, the number of vistal facets is exponential in~$n$, even
within a single orthant \cite{Owen10}.  Thus an efficient method to
move through the vistal facets -- or prune the list of relevant ones
-- would likely improve calculation time of the mean.
\end{rem}

%%%%%%%%%%%%%%%%%%%%%%%%%%%%%%%%%%%%%%%%%%%%%%%%%%%%%%%%%%%%%%%%%%%%%%
\section{Computing the mean in tree space}\label{s:computing}%%%%%%%%%
%%%%%%%%%%%%%%%%%%%%%%%%%%%%%%%%%%%%%%%%%%%%%%%%%%%%%%%%%%%%%%%%%%%%%%

Although the algorithm to calculate the mean in
Section~\ref{sub:sturmAlg} is simple and seems to perform well for
small data sets, Remark~\ref{r:precision} indicates that its
convergence rate is sublinear, so in theory it is a poor iterative
method.
% for finding the mean.
This section outlines a general framework for a descent method to find
the mean of a set $T^1,\ldots,T^r$ of $n$-trees.
%, at least at the local level.
It promises to accelerate the convergence considerably by generalizing
powerful nonlinear programming techniques to apply to optimization in
tree space.

%%%%%%%%%%%%%%%%%%%%%%%%%%%%%%%%%%%%%%%%%%%%%%%%%%%%%%%%%%%%%%%%%%%%%%
\subsection{Optimality criteria}\label{sub:optimality}%%%%%%%%%%%%%%%%

We start by analyzing the variance function $S(x)$ of a variable point
$X\in\cT_n$ whose components are represented by the variable vector
$x$.  For $\ell=1,\ldots,r$, let $\geod_\ell$ be the geodesic from $X$
to $T_\ell$, with associated support pair $(\cA^\ell,\cB^\ell)$.  By
summing the lengths $L(\geod_{\ell})$ of these geodesics as given by
Eq.~\eqref{eq:pathlength2}, write the variance in~$\cT_n$~as
\begin{equation*}
  S(x)
  =
  \sum_{\ell=1}^r (L(\geod_{\ell}))^2
  =
  \sum_{\ell=1}^r
    \left[
    \sum_{i=1}^{k_\ell}\big(\norm{A^\ell_i}+\norm{B^\ell_i}\big)^2
    \right]
\end{equation*}
with its derivative given by Eq.~\eqref{eq:partialS}.  Consider this
function in its $\cT^2_n$-version $S^2(\xi)$ as given in
Definition~\ref{d:squaring}.  Using the notation
$$%
  \bar \xi^\ell_i= \sum_{e\in A^\ell_i}\xi_e
$$
and
$$%
\delta^\ell_i
  = \left\{\begin{array}{@{}ll}
    +1&\text{if }A^l_i\text{ and }B^l_i\text{ are disjoint}\\[.2em]
    -1&\text{if }A^l_i = B^l_i\text{ are made up of common edges},\\
    \end{array}\right.
$$
then the corresponding pullback function for $\xi \in \cT_n^2$ can be
derived from Eqs.~\eqref{eq:S} and~\eqref{d:var}:
\begin{equation}\label{eq:squared variance formula}
  S^2(\xi)
  =
  \sum_{\ell=1}^r\sum_{i=1}^{k_\ell} \Big(\delta^\ell_i\sqrt{\bar
  \xi^\ell_i} + \norm{B^\ell_i}\Big)^2.
\end{equation}
If $i(e,\ell)$ denotes the index of the set~$A^\ell_i$ containing~$e$,
then the gradient of~$S^2$ can be obtained from Eq.~\eqref{eq:squared
variance formula}:
\begin{equation*}
  \frac{\del S^2(\xi)}{\del \xi_e}
  =
  \sum_{\ell=1}^r \Bigg(1+\delta^\ell_{i(e,\ell)}
  \frac{||B^\ell_{i(e,\ell)}||}{\sqrt{\bar \xi^l_{i(e,\ell)}}}\Bigg).
\end{equation*}

The differentiability of~$S$ transfers to~$S^2$, as well.

\begin{cor}\label{t:squarediff}
The function $S^2(\xi)$ is continuously differentiable on the interior
of every maximal orthant~$\Or$.
\end{cor}
\begin{proof}
The inverse of the squaring map is continuously differentiable on the
interior of~$\Or$.  Now apply Theorem~\ref{t:diff}.
\end{proof}

The function $S^2(\xi)$ is not necessarily convex on~$\cT_n^2$.  By
Proposition~\ref{p:convex}, however, it does have a unique local
minimum, which is therefore the mean.  Consequently, optimality
conditions for the function~$S^2(\xi)$ on~$\cT_n^2$ can be based on
its behavior in \emph{any one}\/ of the multivistal facets in which
$\xi$ lies.  In particular, we have the following important result.

\begin{cor}\label{c:unconstrained_mean}
The squared image $\ol \cX$ of the Fr\'echet mean $\ol X$ must satisfy
$\nabla S^2(\ol\cX) = 0$ on its orthant $\Or(\ol\cX)$.  If $\ol\cX$
lies interior to a maximal orthant~$\Or$, then $\ol\cX$ is the squared
image of the mean if and only if the gradient satisfies\/ $\nabla
S^2(\ol\cX) = 0$.  These statements are true regardless on which
multivistal facet of $\Or$ the variance function is derived.
\end{cor}

\begin{proof}
Since by Corollary~\ref{t:squarediff} the gradient is independent of
which vistal facet it is calculated from, the gradient $\nabla S^2(\ol
\cX)$ must be zero on any of them in order $\ol\cX$ to be optimal.
Conversely, since $S^2$ attains a unique minimum on~$\cT^2_n$, it
follows that $X$ must be the mean whenever $\nabla S^2(\ol\cX) = 0$ on
an entire maximal orthant.
\end{proof}

\begin{rem}\label{nondifferentiability}
When a point $\cX$ lies on the boundary of a maximal orthant, the
gradient $\nabla S^2(\cX)$ may be zero on $\Or(\cX)$ even if $\cX$ is
not the squared mean, since there may be a maximal orthant $\Or\supset
\Or(\cX)$ having a point with smaller variance than $\cX$.  Finding
$\Or$ from $\cX$ can be quite difficult, since of~$\partial
S^2(\cX)/\xi_e$ may be undefined or infinite for $e\notin \Or(\cX)$.
Furthermore, directional derivatives may fail to be continuous along
orthant boundaries.  This issue presents serious optimization
difficulties in locating sample means, since there is ample evidence
that reasonably evenly distributed data in tree space yield means that
are likely to occur on orthant boundaries, or indeed, even to lie at
the origin; see Section~\ref{sub:stickiness}.  Thus optimality
conditions for the mean when it occurs on orthant boundaries is an
important topic of further~research.
\end{rem}

%%%%%%%%%%%%%%%%%%%%%%%%%%%%%%%%%%%%%%%%%%%%%%%%%%%%%%%%%%%%%%%%%%%%%%
\subsection{A descent method to compute the mean}\label{sub:descent}%%

In spite of Remark~\ref{nondifferentiability}, we can suggest a basic
method for finding the mean in tree space.  The general idea is to
start with some feasible tree, and construct a sequence of trees whose
variance function is decreasing, until arriving at the mean tree.

\begin{alg}[Descent method for computing the mean]\label{a:descent}
\end{alg}
\begin{alglist}
\routine{input}%
Trees $T^1,\ldots,T^r$ in $\cT_n$

\routine{output}%
The mean tree for $T^1,\ldots,T^r$

\routine{initialize}%
Choose some good starting point $\xi^0 \in \cT_n^2$, for example, by
running Sturm's algorithm for a predetermined number of iterations.

\routine{while} the mean has not been found:
\routine{do}
\begin{enumerate}\setlength\itemsep{-.5ex}
\item%
Find the set $\cM$ of all maximal orthants containing $\xi^t$.
\item%
For each $\cV\in\cM$, choose a point $u^0_\cV$ in the interior of
$\cV$.
\item%
Use a nonlinear interior point/penalty function method to find a local
minimum $u^*_\cV$ of $S^2$ in $\cV$.
\item%
If $u^*_\cV\neq\xi^t$ for any $\cV\in\cM$, then choose the $u^*_\cV$
with minimum $S^2(u^*_\cV)$, and set $\xi^{t+1}=u^*_\cV$.
\end{enumerate}
\routine{end} \textsc{while-do}
\routine{return} $\xi^t$\\
\end{alglist}

The local minimum search in Step 3 should be both straightforward and
reasonably fast, and the accuracy of the points $\xi^k$ as
representing the true local minimum of course depends upon the method
used to find it.  Since the function $S^2$ is continuously
differentiable on all ${\cal V}\in{\cal M}$, the search in fact finds
a local minimum on the orthant~$\cal V$.  Since all neighboring
orthants are searched from $\xi^t$, it follows that whenever all of
these local searches converge back to $\xi^t$ then $\xi^t$ must
necessarily be the mean.  Finally, the algorithm terminates after a
finite number of iterations, since no two $\xi^t$ in the sequence can
lie in the same orthant.  The number of iterations depends both on the
number of iterations $t$ and also the size of $\cal M$, each of which
may be exponentially large.  Thus it is important for the
implementation that a good starting point $\xi^0$ be found, and that a
good method be used to determine descent directions in the set of
maximal orthants adjacent to the point~$\xi^t$.  In general, better
local search techniques and starting solutions, perhaps through a
hybrid of Sturm's Algorithm and descent methods, could improve the
accuracy and reliability of procedures to calculate~the~mean.

%%%%%%%%%%%%%%%%%%%%%%%%%%%%%%%%%%%%%%%%%%%%%%%%%%%%%%%%%%%%%%%%%%%%%%
\section{Properties and applications of the mean}\label{s:apps}%%%%%%%
%%%%%%%%%%%%%%%%%%%%%%%%%%%%%%%%%%%%%%%%%%%%%%%%%%%%%%%%%%%%%%%%%%%%%%

This section contains a series of remarks, results, and computational
studies related to the Fr\'echet mean in tree space.

All synthetic examples in this section will be given using the trees in 
Figure~\ref{f:mean_examples}.  This figure depicts three adjacent
 orthants in $\cT_4$, "flattened out" into the plane, to 
make the visualization easier. The edges $e_1'$ and $e_2$ are not
 compatible, so the $(e_1',e_2)$-orthant (shaded in 
 Figure~\ref{f:mean_examples}) is not part of $\cT_4$.
 For tree $T^1$
(respectively, trees $T^2$ and $T^3$), we specify its interior edge
lengths by a pair of coordinates $(e_1,e_2)$ (respectively,
$(e_1,e_2')$ and $(e_1',e_2')$).  The geodesic between any pair of
these trees is a straight line, unless it would cross the shaded region, in which 
case the geodesic is the \emph{cone path}, consisting of
the two legs joining the given points to the origin.
Likewise, $\ol
T$ is the Euclidean barycenter unless it lies in the shaded region, in which case
$\ol T$ is the point on the boundary of the
shaded region that minimizes the sum of the squared geodesic distances
to the three trees.

%This section contains a series of remarks, results, and computational
%studies related to the Fr\'echet mean in tree space.  Several examples
%in this section use the trees given in Figure~\ref{f:mean_examples}.
%The three orthants in~$\cT_4$ there are adjacent, with the shaded
%$(e_1',e_2)$-orthant missing, since $e_1'$ and $e_2$ are not
%compatible.  The orthants have been drawn ``flattened out'' in the
%same plane to make the geodesics and means easier to see, although
%% they are actually in 4 dimensions in tree space.
%$\cT_4$ is not globally embeddable in the plane.  For a tree $T^1$
%(respectively, trees $T^2$ and $T^3$), we specify its interior edge
%lengths by a pair of coordinates $(e_1,e_2)$ (respectively,
%$(e_1,e_2')$ and $(e_1',e_2')$).  The geodesic between any pair of
%these trees is a straight line whenever it does not cross into the
%shaded region, and otherwise it is the \emph{cone path} consisting of
%the two legs joining the given points to the origin.  Likewise, $\ol
%T$ is the Euclidean barycenter whenever that point does not lie in the
%shaded region, and otherwise it is the point on the boundary of the
%shaded region that minimizes the sum of the squared geodesic distances
%to the three trees.

\begin{figure}[ht]
$$%
  \includegraphics[scale=0.45]{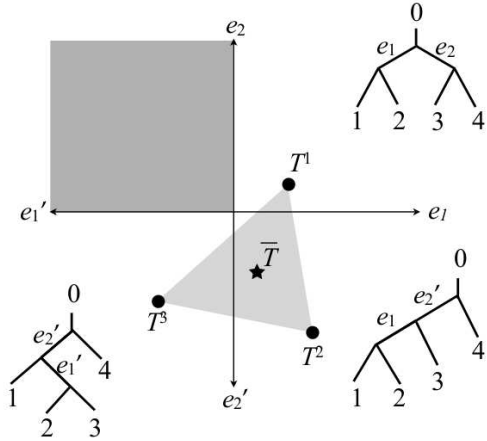}
$$
\caption{Example for the remarks.}\label{f:mean_examples}
\end{figure}

%%%%%%%%%%%%%%%%%%%%%%%%%%%%%%%%%%%%%%%%%%%%%%%%%%%%%%%%%%%%%%%%%%%%%%%
\subsection{Composition of the mean tree}\label{r:edges_in_mean}%%%%%%

The topology of the mean tree depends on both the topologies and the
edge lengths of the sample trees.  Consider, for example, trees
$T^1=(3,1)$ and $T^3=(1,3)$.  The mean between these two trees is the
midpoint $\ol T^2 = (1,1)$ of the segment joining them.  Changing both
edge lengths of~$T^1$ to~$5$, however, yields a midpoint $\ol T^1 =
(1,2)$; similarly, by symmetry, changing both edge lengths of~$T^3$
to~$5$ yield the midpoint $\ol T^3 = (1,2)$.  That said, in general we
can give some indication of what edges lie in the mean tree.

\begin{lem}
\label{l:edges_in_mean}
Every edge of the mean tree is an edge of some sample tree.
Furthermore, if an edge appears in all sample trees, then it must also
appear in the mean tree.
\end{lem}
\begin{proof}
If a tree contains an edge not in any sample tree, then contracting
this edge gives a tree with a smaller variance function.  Now suppose
that the edge~$e$ is contained in all sample trees, and thus is
compatible with all other edges in the sample trees.  Since the mean
contains only edges from the sample trees, edge~$e$ is also compatible
with all edges in the mean tree.
Thus if the mean tree does not contain $e$, we can add in $e$ with 
length equal to the minimum of its lengths in the sample trees, yielding a 
tree that is closer to all the sample trees, which is a contradiction.
\end{proof}

\subsection{Other notions of consensus tree}\label{r:mean_defs}%%%%%%%

Several authors have proposed notions of ``center'' for a set ${\bm T}
= \{T^1,\ldots,T^r\}$ of points in~$\cT_n$.  The Euclidean or
combinatorial properties of these centers make them useful for
representing consensus trees.  Here we compare three such centers
with the Fr\'echet mean~$\ol T$ of~$\bm T$.  All of these centers agree when $\bm T$
lies entirely in a single orthant of~$\cT_n$, but fail to agree for
more globally distributed samples from tree space.

\begin{ex}[\bf The majority-rule consensus (MRC) tree]
First introduced by Margush and McMorris \cite{MargushMcMorris81},
this is the tree whose edge set is comprised of those edges that
appear in at least half of the trees in $\bm T$.  It, or an variation, is widely
 used in the phylogenetics literature.  The topology of $\ol T$ is not a 
 refinement of the MRC tree, unlike many other consensus methods \cite{Bryant03}.
 For example, consider the trees in Figure~\ref{f:mean_examples}
 with coordinates $T^1 = (1,1)$, $T^2 = (1,1)$, and $T^3 = (5,6)$.  
 The mean of these trees is the Euclidean barycenter $\ol T^3 = (1,2)$, while the MRC tree
 has the topology of tree $T^2$, so neither tree is a refinement of the other.
 \end{ex}

\begin{ex}[\bf Sturm's inductive mean]\label{e:indMean}
The inductive mean (Definition~\ref{d:indMean}) of the set~$\bm T$,
for some ordering of~$\bm T$, does not coincide with~$\ol T$, and it
can differ depending upon the ordering.  Consider the trees in
Figure~\ref{f:mean_examples} with coordinates $T^1 = (3,10)$, $T^2 =
(3,3)$, and $T^3 = (10,3)$.  Either order having $T^1$ and $T^3$ first
yields the inductive mean $\tilde T^2 = (1,1)$.  Either order having
$T^1$ and $T^2$ first yields the inductive mean $\tilde T^3 =
(0.390,0.117)$, and either order having $T^2$ and $T^3$ first yields
the inductive mean $\tilde T^1 = (0.117,0.390)$.  These have different
topologies, and none of them equals~$\ol T$, which has all edges~$0$.
\end{ex}

\begin{ex}[\bf The BHV centroid]\label{e:BHVcentroid}
Billera, Holmes, and Vogtmann \cite{BHV01} define the \emph{centroid}
of $\bm T = \{T^1,\ldots,T^r\}$ inductively on $r$.  For $r = 2$, the
centroid is the midpoint of the two trees.  For $r > 2$, the centroid
is obtained as follows: set $\bm T^1 = \bm T$ and inductively find the
centroid of each subset of $r-1$ trees in $\bm T^1$ to obtain a new
set $\bm T^2$ of $r$ trees. Repeat this process on the new set,
creating a sequence $\bm T^1,\bm T^2,\ldots$ of $r$-sets of trees.
The BHV centroid of~$\bm T$ is the limit $\hat T$ of any sequence of
points chosen from each of the sets~$\bm T^i$.  This process converges
in a general global NPC space \cite[Theorem~4.1]{BHV01}.

Billera, Holmes, and Vogtmann note that in Euclidean space, the
centroid and Fr\'echet mean coincide.  This is not generally the case in tree
space.  Consider, for example, the trees in Figure~\ref{f:mean_examples} with
coordinates $T^1 = (2,4)$, $T^2 = (2,2)$, and $T^3 = (4,2)$.  Then
$\ol T$ is again the origin, while it is easy to see that the BHV centroid
lies off the origin.
\end{ex}

\subsection{Stickiness of the mean}\label{sub:stickiness}%%%%%%%%%%%%%

Sullivant \cite{Sullivant10} noticed the tendency of the Fr\'echet
mean to be \emph{sticky}, which in this context means that perturbing
one or more of the trees in the set $\bm T$ does not necessarily
change any of the coordinates of~$\ol T$.  Take, for
example, the points $T^1 = (3,10)$, $T^2 = (3,3)$, and $T^3 = (10,3)$.
The mean $\ol T$ lies at the origin, and remains there even if the
coordinates of any of the three trees~$T^i$ are perturbed even up to a
full unit.  Sticky means occur exclusively on orthants of lower
dimension, underscoring the importance of closely investigating
properties of mean trees that lie on orthant boundaries.

The notion of stickiness has been quantified via a Central Limit
Theorem for means of probability distributions on certain NPC spaces
\cite{basrak,stickyCLT}.

%%%%%%%%%%%%%%%%%%%%%%%%%%%%%%%%%%%%%%%%%%%%%%%%%%%%%%%%%%%%%%%%%%%%%%
\subsection{Application to biological data}\label{sub:data}

Statistical applications of this research are important in several
areas of mathematics, biology, and medicine.  Here, we consider a
well-studied data set in phylogenetics
with respect to the Fr\'echet mean.  For applications of the Fr\'echet 
mean to medical imaging, see \cite{SkwererJMIV13} and \cite{IPMI2013}.

\begin{ex}[\bf Gene trees vs.\ species trees]\label{e:gene}
% Applications to yeast genes
% Mean trees for biological data sets can make sense biologically.
A \emph{gene tree} is a phylogenetic tree representing the
evolutionary history of a particular gene found in some set of species.  In contrast,
a \emph{species tree} is a phylogenetic tree representing the
evolutionary history of the species themselves: the history of population
bifurcations leading to divergence.  Due to natural processes such as
incomplete lineage sorting, gene trees for different genes can have
different topologies, even when sampled from the same set of
individuals---let alone the same set of species---and hence a gene
tree need not share its topology with the species tree (see
\cite{Maddison97}, for example).  Furthermore, the most likely gene
tree topology need not agree with the species tree topology
\cite{DegnanRosenberg06}.  However, species trees are usually
reconstructed from gene trees, and a major open question is how best
to accomplish~this.

We examined the yeast data set of Rokas et al.~\cite{Rokas03}.  For
eight species of yeast, they identified 106 genes and reconstructed
the corresponding gene tree with edge lengths for each using a maximum
likelihood approach.  In these 106 gene trees, there were 21 different
topologies.  We used Sturm's algorithm to compute the Fr\'echet mean of these
gene trees.  This mean tree had the same topology as the agreed-upon
species tree \cite{EdwardsLiuPearl07}.  In general, the mean gene tree does
not necessarily identify the species tree, as a consequence of stickiness,
when branch lengths are taken into account \cite{Sullivant10}; that is,
 two finite samples of gene trees can
yield the same mean tree but have different species trees.  However,
we conjecture that the topology of the species tree is a refinement of
the topology of the Fr\'echet mean of the gene trees.  That is,
stickiness of the Fr\'echet mean forces some edges to have zero length
but should not add any extraneous edges to this~mean.
\end{ex}

%%%%%%%%%%%%%%%%%%%%%%%%%%%%%%%%%%%%%%%%%%%%%%%%%%%%%%%%%%%%%%%%%%%%%%
\section{Globally nonpositively curved spaces}\label{s:npc}%%%%%%%%%%%
%%%%%%%%%%%%%%%%%%%%%%%%%%%%%%%%%%%%%%%%%%%%%%%%%%%%%%%%%%%%%%%%%%%%%%

Virtually all of our treatment of tree spaces extends to more general
global NPC spaces.  This section reframes the concepts and notation of
the paper in the context of global NPC spaces, particularly orthant
spaces, and shows how the results of the paper generalize to these
spaces.

%%%%%%%%%%%%%%%%%%%%%%%%%%%%%%%%%%%%%%%%%%%%%%%%%%%%%%%%%%%%%%%%%%%%%%
\subsection{The geometry of nonpositively curved spaces}\label{ss:geomNPC}

Fix a metric space~$\cT = (\cT,d)$.  A~\emph{path} in~$\cT$ is the
image of a continuous map $\gamma: [0,1] \to \cT$.  Write
$\gamma_\lambda = \gamma(\lambda)$ for $0\leq \lambda \leq 1$.  The
\emph{length} of $\gamma$ is the supremum of all sums
$$%
  d(\gamma_{x_0},\gamma_{x_1}) +
  d(\gamma_{x_1},\gamma_{x_2}) + \cdots +
  d(\gamma_{x_{k-1}},\gamma_{x_k})
$$
such that $0 \leq x_0 \leq \cdots \leq x_k \leq 1$.  A path is a
\emph{(global) geodesic} if the distance $d(\gamma_x,\gamma_y)$
between any pair of points on $\gamma$ equals the length of that
portion of $\gamma$ between them.  A~\emph{geodesic space} is a
complete metric space such that every pair $\{x,y\}$ of points is
joined by a path $\gamma$ whose length is the distance~$d(x,y)$
between $x$ and~$y$.

\begin{defn}\label{d:npc}
A metric space $(\cT,d)$ is \emph{globally nonpositively curved}, also
known as \emph{global NPC} or \emph{CAT(0)}, if for every triple of
points $a,b,c \in \cT$, any point $x$ on a geodesic joining
$a$~to~$b$, and any \emph{reference triangle}\/ $a'b'c'$ in Euclidean
space with edge lengths $d(a,b)$, $d(b,c)$, and $d(a,c)$, the unique
point $x'$ on $a'b'$ at distance $d(a,x)$ from~$a'$ satisfies $d(x,c)
\leq ||x'-c'||$.
\end{defn}

The definition essentially says that triangles created by joining
points by geodesics in a global NPC space are ``skinnier'' than their
counterparts in Euclidean space.

\begin{lem}[{\cite[Proposition~2.3]{sturm}}]\label{l:geodesics}
In a global NPC space every pair of points is joined by a unique
geodesic.\qed
\end{lem}

A real-valued function $f: \cT \to \RR$ is \emph{convex} if $f \circ
\gamma$ is convex for all geodesics~$\gamma$; that is, if
\begin{equation}\label{e:fconvex}
  f(\gamma_\lambda) \leq (1-\lambda)f(\gamma_0) + \lambda f(\gamma_1)
\end{equation}
for all geodesics $\gamma: [0,1] \to \cT$.

\begin{ex}\label{e:dconvex}
For any point $t \in \cT$, the distance $d_t(x) = d(x,t)$ from a point
$x\in \cT$ to $t$ is a convex function of $x$ \cite[Corollary~2.5 and
subsequent Remark~(i)]{sturm}.
\end{ex}

A real-valued function $f: \cT \to \RR$ is \emph{strictly convex} if
Eq.~\eqref{e:fconvex} holds strictly for $0 < \lambda < 1$.

\begin{lem}[{\cite[Proposition~1.7 and Remark~1.8]{sturm}}]\label{l:min}
Any strictly convex continuous function on a global NPC space attains
a unique minimum.\qed
\end{lem}

\begin{cor}\label{c:strictly}
If $\bm T = \{t^1,\ldots,t^r\}$ is a set of points in $\cT$, and $f:
\RR^r_+ \mapsto \RR$ is any (strictly) convex function, then the
function $F:\cT\mapsto \RR$ defined by
$$%
  F(x) = f(d_{t^1}(x),\ldots,d_{t^r}(x))
$$
is a (strictly) convex function.
\end{cor}

In particular, the variance function for a set $\bm T$ of points is a
convex function and hence attains a unique minimum at the Fr\'echet
mean.

%%%%%%%%%%%%%%%%%%%%%%%%%%%%%%%%%%%%%%%%%%%%%%%%%%%%%%%%%%%%%%%%%%%%%%
\subsection{Means and variances in global NPC spaces}\label{ss:NPCmean}

This subsection generalizes the notion of mean and variance to general
probability measures in global NPC spaces.  The results follow from
those of Sturm \cite{sturm} in this area.  Let $\cP(\cT)$ be the set
of probability measures on a global NPC space~$\cT$.  If $\rho \in
\cP(\cT)$ is such a measure, then its \emph{variance} is
$$%
  \var(\rho) = \inf_{x\in\cT} \int_\cT d^2(x,y)\rho(dy).
$$
The variance can be infinite in general, but not in the case of most
interest to us, when $\rho$ has finite support, meaning that there is
a set $\bm T = \{t^1,\ldots,t^r\}$ of points in $\cT$, along with
nonnegative weights $\omega_1,\ldots,\omega_r$ satisfying $\omega_1 +
\cdots + \omega_r = 1$, such that the point~$t_i$ has mass $\rho(t_i)
= w_i$ for $i = 1,\ldots,r$.
% and $\rho(x) = 0$ otherwise.
% \comment{EM: specifying $\rho(x) = 0$ otherwise is redundant, since
% $\rho$ is a probability measure and weights sum to~$1$.
Let $\cP^2(\cT)$ be the set of measures in $\cP(\cT)$ having finite
variance.

\begin{prop}[{\cite[Proposition~4.3]{sturm}}]\label{p:var}
For a global NPC space $\cT$ and probability measure $\rho \in
\cP^2(\cT)$, there is a unique point $\ol \rho \in \cT$ such that
$\var(\rho) = \int_\cT d^2(\ol \rho,y)p(dy)$.
\end{prop}

The point $\ol \rho$ is referred to as the \emph{Fr\'echet mean}\/ or
\emph{barycenter}\/ in this context as well, and when $\rho$ has
finite support with $\omega_i = \frac 1r$ for all~$r$, it is a direct
generalization of the definition of mean given in
Section~\ref{s:mean}.  The notion of inductive mean given by
Definition~\ref{d:indMean} extends easily to an arbitrary global NPC
space~$\cT$, and the following result generalizes
Theorem~\ref{t:sturmPoints}.

\begin{thm}[{\cite[Theorem~4.7]{sturm}}]\label{t:sturm}
For a global NPC space $\cT$ and probability measure $\rho \in
\cP^2(\cT)$, let $X^1,X^2,\ldots$ be a sequence of independent and
identically distributed random variables drawn from $\rho$.  Then with
probability~$1$, the sequence of inductive mean values
$\mu_1,\mu_2,\ldots$ approaches the mean $\ol \rho$ of $\rho$.
%; that
%is,
%$$%
%  \frac1k\sum_{\ell=1,\ldots,k}^{\longrightarrow}X^\ell \to \ol \rho.
%$$
\end{thm}

\begin{cor}\label{c:weighted}
The convergence properties of the sequence of inductive means given by
Algorithm~\ref{a:sturm} continue to hold on any probability
distribution $\rho \in \cP^2(\cT)$, by sampling the points of $\cT$
according to the specified distribution.
\end{cor}

Note that Corollary~\ref{c:weighted} was independently observed by
Ba\v c\'ak~\cite{Bacak12}.

\begin{rem}\label{r:LLN}
As an application of Corollary~\ref{c:weighted}, Markov chain Monte
Carlo (MCMC) simulations produce phylogenetic trees sampled
independently from a fixed finite-variance distribution on the entire
space $\cT_n$.  Calculating inductive mean values of repeated samples
from this distribution results in a method to approximate the mean of
the distribution.
\end{rem}

%%%%%%%%%%%%%%%%%%%%%%%%%%%%%%%%%%%%%%%%%%%%%%%%%%%%%%%%%%%%%%%%%%%%%%
\subsection{NPC orthant spaces}\label{sub:orthant}%%%%%%%%%%%%%%%%%%%%

\begin{defn}\label{d:orthant}
The \emph{orthant space} $\Or(\cE,\Omega)$ consists of a set $\cE$ of
\emph{axes}\/ together with a simplicial complex $\Omega\subseteq
2^\cE$, called the \emph{scaffold complex}.  Two elements of $\cE$ are
\emph{compatible}\/ if they appear in some face of $\Omega$.  Each
face $F\in\Omega$ is associated with a copy $\Or_F$ of $\RR^F_+$, the
\emph{orthant}\/ associated with $F$.  The orthant space
$\Or(\cE,\Omega)$ is the union of the orthants $\Or_F$ for
$F\in\Omega$, with points identified whenever their nonzero
coordinates agree on all elements~of~$\cE$.
\end{defn}

An orthant space can be thought of as constructed by gluing together
orthants according to instructions laid out by the scaffold complex,
and in fact the scaffold complex is (homeomorphic to) the \emph{link}
of the origin in the orthant space $\Or(\cE,\Omega)$.

\begin{ex}\label{e:tree}
Tree space $\cT_n$ is an orthant space: $\cE$ corresponds to the set
of splits on $\{0,\ldots,,n\}$, and $\Omega$ corresponds to the
collection of sets of splits that are compatible in the sense of
Section~\ref{ss:treespace}.
\end{ex}

A \emph{path}\/ in an orthant space $\Or(\cE,\Omega)$ is defined as in
Section~\ref{ss:geomNPC}.  A locally length-minimizing path is a
\emph{geodesics}, which always consists of a finite number of linear
legs through intermediate orthants of~$\cT$, as in the case of tree
space (Section~\ref{ss:geodesics}).  As with tree space,
$\Or(\cE,\Omega)$ is always path-connected.

Although any orthant space is geodesic, it may not be global NPC.

\begin{ex}
The space $\cT=\Or(\cE,\Omega)$, where $\cE$ is indexed by $\{1,2,3\}$
and the scaffold complex $\Omega$ has facets
% $\{\nothing,1,2,3,12,13,23\}$,
$\{1,2\}$, $\{1,3\}$, and $\{2,3\}$ is not global NPC.  Indeed, the
two points $x = (1,0,0)$ and $y = (0,1,1)$ in~$\cT$ have a pair of
geodesics between them, namely $[(1,0,0), (0,1/2,0)] \cup [(0,1/2,0),
(0,1,1)]$ and $[(1,0,0), (0,0,1/2)] \cup [(0,0,1/2), (0,1,1)]$.  By
Lemma~\ref{l:geodesics}, $\cT$ cannot be global NPC.
\end{ex}

% Billera, Holmes, and Vogtmann [{\cite[Lemma~4.1]{BHV01}}]
M.\thinspace{}Gromov \cite{gromov87} determined conditions on~$\Omega$
that characterize when $\Or(\cE,\Omega)$ has nonpostive curvature (in
fact, Gromov worked with arbitrary cubical complexes), based on the
following standard notion from geometric combinatorics.

\begin{defn}
The simplicial complex $\Omega$ is \emph{flag} if $F \in \Omega$
whenever all pairs of elements in~$F$ are compatible.
\end{defn}

\begin{prop}[{\cite{gromov87}}]\label{p:scaffold}
An orthant space $\Or(\cE,\Omega)$ is global NPC if and only if
$\Omega$ is flag.
\end{prop}

In particular, tree space~$\cT_n$ is global NPC, since its scaffold
complex is defined precisely by the pairwise compatibility between its
splits (this is the proof given in \cite{BHV01}).  Generally, any
global NPC orthant space can be defined entirely by its set of
compatible elements.

\begin{defn}
The \emph{scaffold graph} $\cG(\cE,\Omega)$ of an orthant space
$\Or(\cE,\Omega)$ is the graph with vertex set~$\cE$ whose edges are
the pairs of compatible elements of~$\cE$.
\end{defn}

\begin{lem}
The orthants of a global NPC orthant space $\Or(\cE,\Omega)$ are
precisely the \emph{clique sets}\/ (sets of mutually compatible edges)
of the scaffold graph~$\cG(\cE,\Omega)$.\qed
\end{lem}

Thus there is a one-to-one correspondence between orthant spaces and
graphs.  A general global NPC orthant space $\Or(\cE,\Omega)$ need not
have all of its maximal orthants the same dimension, since maximal
orthants correspond to the maximal cliques in $\cG(\cE,\Omega)$.  The
dimension of the maximal orthants, however, is not relevant to any of
the results in the previous sections, except when the dimension is
given explicitly.

\begin{ex}
The space of trees in which each split is associated with an
$m$-dimensional vector instead of a single length is an NPC orthant
space.  In this case, the scaffold graph is the scaffold graph of tree
space $\cT_n$, with each vertex replaced by $K_m$, the complete graph
on $m$ vertices.  Our software implementation also computes geodesics
and means in this space.
\end{ex}

\begin{ex}\label{LAN}
The generality of scaffold graphs to define any global NPC orthant
space provides an opportunity to extend the statistical structures of
this paper to a wider range of applications.  As one example, consider
a computer network specified by its computational devices and the
graph $\cG$ denoting those pairs of computers that are compatible with
each other.  A \emph{local area network (LAN)}\/ for this system is a
set $C$ of mutually compatible computers---that is, a clique of $\cG$.
A \emph{local network configuration (LNC)}\/ is a LAN $C$ together
with a measure $w_e$ of participation of each computer $e\in C$ in the
LAN $C$.  Some important areas of analysis of the network $\cG$ might
be the relationship between the LNCs associated with $\cG$, in terms
of the number and participation weight of common computers and the
relative compatibility of the noncommon computers (although it does
\emph{not}\/ model chaining-related measures such as the number of
nodes in a communications path).  The global NPC orthant space
generated by $\cG$ would be a good framework for answering questions
like this associated with the LANs of the network.
\end{ex}

The combinatorics of geodesics in Section~\ref{s:treespace}
generalizes immediately to global NPC orthant spaces, using the
generalized notation in this section.  None of the proofs in
Sections~\ref{s:treespace}--\ref{s:computing} rely on
% peculiarities
particulars of tree space except the flag property.  Thus we have the
following.

\begin{cor}\label{c:npc}
The results in Sections~\ref{s:treespace}--\ref{s:computing} (except
for statements specifying dimension) extend to arbitrary global NPC
orthant spaces, using the definitions in this section.  In particular,
the GTP algorithm \cite{OwenProvan10} for finding geodesics in tree
space, Sturm's Algorithm (Algorithm~\ref{a:sturm}), and the Descent
Method (Algorithm~\ref{a:descent}) apply in the more general setting
of global NPC orthant spaces.
\end{cor}

% \comment{SP: I tried to figure out a way to generalize
% Theorem~\ref{t:diff} to more general functions} The only reasonable
% function that does this is of the form
% $F(x)=f(d_{t^1}(x),\ldots,d_{t^r}(x))$, where
% $$%
%   f(d) = \sum_{i=1}^r d_i^2g_i(d_i)
% $$
% with $g_i$ a continuously differentiable functions of the distance
% $d_i = d_{t^i}$.  This doesn't look too interesting, although it does
% include finding the point of minimum $k$th moment of a population, for
% whatever that is worth.  \comment{SP: Is this worth including here?
% MO: probably not.}

\begin{rem}\label{r:summary}
The results here extend even further.  For example, Ardila, Owen, and
Sullivant~\cite{ArdilaOwenSullivant} extend the global NPC theory, and
in particular the GTP algorithm, to the case of \emph{cubical
complexes}, where orthants are replaced by Euclidean cubes.  Sturm's
algorithm extends to these global NPC cubical complexes, and there is
every reason to believe that the idea of vistal cells and the Descent
Method can be extended as well, although without the polyhedrality.
Furthermore, the results extend to negative edge lengths, as described 
in Remark~\ref{r:neg_edges}.
\end{rem}  

\begin{rem}\label{r:neg_edges}
Negative values can also be allowed for the coordinates in a global NPC orthant
space.  In this case, it remains a global NPC space, and the results listed in Corollary~
\ref{c:npc} for NPC orthant spaces also hold here, with the following modification.
If a negative value appears in a common 'split', then that negative value is used 
in the geodesic calculations.  However, if it appears in a split that is not in common,
then its absolute value is used in the geodesic calculations.  
\end{rem}

\end{document}

%% file: tree_examples.pstex_t
\begin{picture}(0,0)%
\includegraphics{tree_examples.pstex}%
\end{picture}%
\setlength\unitlength{3947sp}%
\begingroup\makeatletter\ifx\SetFigFont\undefined%
\gdef\SetFigFont#1#2#3#4#5{%
  \reset@font\fontsize{#1}{#2pt}%
  \fontfamily{#3}\fontseries{#4}\fontshape{#5}%
  \selectfont}%
\fi\endgroup%
\begin{picture}(5800,3399)(1,-3035)
\put(5044,-1390){\makebox(0,0)[lb]{\smash{{\SetFigFont{10}{13.2}{\rmdefault}{\mddefault}{\updefault}{\color[rgb]{0,0,0}$|e_4|=4$}%
}}}}
\put(4902,-1704){\makebox(0,0)[lb]{\smash{{\SetFigFont{10}{13.2}{\rmdefault}{\mddefault}{\updefault}{\color[rgb]{0,0,0}$|e_5|=3$}%
}}}}
\put(4366,-563){\makebox(0,0)[lb]{\smash{{\SetFigFont{10}{13.2}{\rmdefault}{\mddefault}{\updefault}{\color[rgb]{0,0,0}$|e_6|=10$}%
}}}}
\put(927,-1562){\makebox(0,0)[lb]{\smash{{\SetFigFont{10}{13.2}{\rmdefault}{\mddefault}{\updefault}{\color[rgb]{0,0,0}$|e_1|=10$}%
}}}}
\put(473,-1128){\makebox(0,0)[lb]{\smash{{\SetFigFont{10}{13.2}{\rmdefault}{\mddefault}{\updefault}{\color[rgb]{0,0,0}$|e_3|=3$}%
}}}}
\put(898,-514){\makebox(0,0)[lb]{\smash{{\SetFigFont{10}{13.2}{\rmdefault}{\mddefault}{\updefault}{\color[rgb]{0,0,0}$|e_2|=4$}%
}}}}
\put(3826,-2986){\makebox(0,0)[lb]{\smash{{\SetFigFont{10}{13.2}{\rmdefault}{\mddefault}{\updefault}{\color[rgb]{0,0,0} splits $\begin{array}{l@{\,}r@{\,}c@{}l}e_4\!:&\{0,1,4,5\}&|&\{2,3\}\\e_5\!:&\{0,1,2,3\}&|&\{4,5\}\\e_6\!:&\{0,1\}&|&\{2,3,4,5\}\end{array}$}%
}}}}
\put(  1,-2986){\makebox(0,0)[lb]{\smash{{\SetFigFont{10}{13.2}{\rmdefault}{\mddefault}{\updefault}{\color[rgb]{0,0,0}splits $\begin{array}{l@{\,}r@{\,}c@{}l}e_1\!:&\{0,1,2,5\}&|&\{3,4\}\\e_2\!:&\{0,3,4,5\}&|&\{1,2\}\\e_3\!:&\{0,5\}&|&\{1,2,3,4\}\end{array}$}%
}}}}
\put(4801,-2461){\makebox(0,0)[lb]{\smash{{\SetFigFont{10}{13.2}{\rmdefault}{\mddefault}{\updefault}{\color[rgb]{0,0,0}$T'$}%
}}}}
\put(976,-2461){\makebox(0,0)[lb]{\smash{{\SetFigFont{10}{13.2}{\rmdefault}{\mddefault}{\updefault}{\color[rgb]{0,0,0}$T$}%
}}}}
\end{picture}%

%% file: picard_queyranne.pstex_t
\begin{picture}(0,0)%
\includegraphics{picard_queyranne.pstex}%
\end{picture}%
\setlength{\unitlength}{3947sp}%
\begingroup\makeatletter\ifx\SetFigFont\undefined%
\gdef\SetFigFont#1#2#3#4#5{%
  \reset@font\fontsize{#1}{#2pt}%
  \fontfamily{#3}\fontseries{#4}\fontshape{#5}%
  \selectfont}%
\fi\endgroup%
\begin{picture}(5527,9407)(-899,-9569)
\put(2525,-7507){\makebox(0,0)[lb]{\smash{{\SetFigFont{11}{13.2}{\rmdefault}{\mddefault}{\updefault}{\color[rgb]{0,0,0}$\bar t$}%
}}}}
\put(2516,-9133){\makebox(0,0)[lb]{\smash{{\SetFigFont{11}{13.2}{\rmdefault}{\mddefault}{\updefault}{\color[rgb]{0,0,0}$\bar s$}%
}}}}
\put(2548,-6978){\makebox(0,0)[lb]{\smash{{\SetFigFont{11}{13.2}{\rmdefault}{\mddefault}{\updefault}{\color[rgb]{0,0,0}$\bar s$}%
}}}}
\put(2556,-4718){\makebox(0,0)[lb]{\smash{{\SetFigFont{11}{13.2}{\rmdefault}{\mddefault}{\updefault}{\color[rgb]{0,0,0}$\bar t$}%
}}}}
\put(2344,-7145){\makebox(0,0)[lb]{\smash{{\SetFigFont{11}{13.2}{\rmdefault}{\mddefault}{\updefault}{\color[rgb]{0,0,0} }%
}}}}
\put(-374,-4486){\makebox(0,0)[lb]{\smash{{\SetFigFont{11}{13.2}{\rmdefault}{\mddefault}{\updefault}{\color[rgb]{0,0,0}(Flow given for intermediate edges only; edges with no numbers have 0 flow.)}%
}}}}
\put(1121,-4329){\makebox(0,0)[lb]{\smash{{\SetFigFont{11}{13.2}{\rmdefault}{\mddefault}{\updefault}{\color[rgb]{0,0,0}(b) Flow graph and associated max flow}%
}}}}
\put( 94,-1497){\makebox(0,0)[lb]{\smash{{\SetFigFont{11}{13.2}{\rmdefault}{\mddefault}{\updefault}{\color[rgb]{0,0,0}(a) Original incompatibility graph with (squared)  edge weights}%
}}}}
\put(2578,-4151){\makebox(0,0)[lb]{\smash{{\SetFigFont{11}{13.2}{\rmdefault}{\mddefault}{\updefault}{\color[rgb]{0,0,0}$\bar s$}%
}}}}
\put(2579,-1856){\makebox(0,0)[lb]{\smash{{\SetFigFont{11}{13.2}{\rmdefault}{\mddefault}{\updefault}{\color[rgb]{0,0,0}$\bar t$}%
}}}}
\put(1464,-9281){\makebox(0,0)[lb]{\smash{{\SetFigFont{11}{13.2}{\rmdefault}{\mddefault}{\updefault}{\color[rgb]{0,0,0}(d) Final acyclic graph $G^*$.}%
}}}}
\put(1225,-7299){\makebox(0,0)[lb]{\smash{{\SetFigFont{11}{13.2}{\rmdefault}{\mddefault}{\updefault}{\color[rgb]{0,0,0}(thick edges are doubly-directed.)}%
}}}}
\put(463,-7142){\makebox(0,0)[lb]{\smash{{\SetFigFont{11}{13.2}{\rmdefault}{\mddefault}{\updefault}{\color[rgb]{0,0,0}(c) Residual graph $G^r$  with contracted nodes circled}%
}}}}
\put(-899,-9511){\makebox(0,0)[lb]{\smash{{\SetFigFont{12}{14.4}{\rmdefault}{\mddefault}{\updefault}{\color[rgb]{0,0,0} $U=\{x_1,x_2,x_3,t_1\}$, $V=\{x_4,x_5,x_6,t_2,t_3,t_4,t_5\}$, $W=\{x_7,t_6\}$, $X=\{x_8,t_7\}$}%
}}}}
\put(751,-1117){\makebox(0,0)[lb]{\smash{{\SetFigFont{9}{10.8}{\rmdefault}{\mddefault}{\updefault}{\color[rgb]{0,0,0}$x_1$}%
}}}}
\put(1269,-1138){\makebox(0,0)[lb]{\smash{{\SetFigFont{9}{10.8}{\rmdefault}{\mddefault}{\updefault}{\color[rgb]{0,0,0}$x_2$}%
}}}}
\put(1755,-1131){\makebox(0,0)[lb]{\smash{{\SetFigFont{9}{10.8}{\rmdefault}{\mddefault}{\updefault}{\color[rgb]{0,0,0}$x_3$}%
}}}}
\put(2266,-1139){\makebox(0,0)[lb]{\smash{{\SetFigFont{9}{10.8}{\rmdefault}{\mddefault}{\updefault}{\color[rgb]{0,0,0}$x_4$}%
}}}}
\put(2769,-1131){\makebox(0,0)[lb]{\smash{{\SetFigFont{9}{10.8}{\rmdefault}{\mddefault}{\updefault}{\color[rgb]{0,0,0}$x_5$}%
}}}}
\put(3279,-1124){\makebox(0,0)[lb]{\smash{{\SetFigFont{9}{10.8}{\rmdefault}{\mddefault}{\updefault}{\color[rgb]{0,0,0}$x_6$}%
}}}}
\put(3796,-1131){\makebox(0,0)[lb]{\smash{{\SetFigFont{9}{10.8}{\rmdefault}{\mddefault}{\updefault}{\color[rgb]{0,0,0}$x_7$}%
}}}}
\put(4305,-1116){\makebox(0,0)[lb]{\smash{{\SetFigFont{9}{10.8}{\rmdefault}{\mddefault}{\updefault}{\color[rgb]{0,0,0}$x_8$}%
}}}}
\put(1005,-441){\makebox(0,0)[lb]{\smash{{\SetFigFont{9}{10.8}{\rmdefault}{\mddefault}{\updefault}{\color[rgb]{0,0,0}$t_1$}%
}}}}
\put(1508,-441){\makebox(0,0)[lb]{\smash{{\SetFigFont{9}{10.8}{\rmdefault}{\mddefault}{\updefault}{\color[rgb]{0,0,0}$t_2$}%
}}}}
\put(2018,-448){\makebox(0,0)[lb]{\smash{{\SetFigFont{9}{10.8}{\rmdefault}{\mddefault}{\updefault}{\color[rgb]{0,0,0}$t_3$}%
}}}}
\put(2512,-449){\makebox(0,0)[lb]{\smash{{\SetFigFont{9}{10.8}{\rmdefault}{\mddefault}{\updefault}{\color[rgb]{0,0,0}$t_4$}%
}}}}
\put(3030,-456){\makebox(0,0)[lb]{\smash{{\SetFigFont{9}{10.8}{\rmdefault}{\mddefault}{\updefault}{\color[rgb]{0,0,0}$t_5$}%
}}}}
\put(3532,-449){\makebox(0,0)[lb]{\smash{{\SetFigFont{9}{10.8}{\rmdefault}{\mddefault}{\updefault}{\color[rgb]{0,0,0}$t_6$}%
}}}}
\put(4042,-449){\makebox(0,0)[lb]{\smash{{\SetFigFont{9}{10.8}{\rmdefault}{\mddefault}{\updefault}{\color[rgb]{0,0,0}$t_7$}%
}}}}
\put(1756,-8346){\makebox(0,0)[lb]{\smash{{\SetFigFont{10}{12.0}{\rmdefault}{\mddefault}{\updefault}{\color[rgb]{0,0,0}$U$}%
}}}}
\put(2245,-8341){\makebox(0,0)[lb]{\smash{{\SetFigFont{9}{10.8}{\rmdefault}{\mddefault}{\updefault}{\color[rgb]{0,0,0}$V$}%
}}}}
\put(2724,-8333){\makebox(0,0)[lb]{\smash{{\SetFigFont{9}{10.8}{\rmdefault}{\mddefault}{\updefault}{\color[rgb]{0,0,0}$W$}%
}}}}
\put(3283,-8326){\makebox(0,0)[lb]{\smash{{\SetFigFont{9}{10.8}{\rmdefault}{\mddefault}{\updefault}{\color[rgb]{0,0,0}$X$}%
}}}}
\end{picture}%